\documentclass[12pt]{article}
\usepackage{amsfonts,amssymb, amsmath, latexsym}

\newtheorem{theorem}{Theorem}[section]
\newtheorem{lemma}[theorem]{Lemma}

\newtheorem{proposition}[theorem]{Proposition}
\newtheorem{definition}[theorem]{Definition}

\newenvironment{proof}
{\par\addvspace{0.3cm}\noindent{\rm Proof. }}
{\nopagebreak\mbox{}\hfill $\Box$\par\addvspace{0.25cm}}

\newcommand{\qed}{\hfill $\Box$}
\newcommand{\be}{\begin{equation}}
\newcommand{\ee}{\end{equation}}

\newcommand{\bq}{\begin{eqnarray}}
\newcommand{\eq}{\end{eqnarray}}
\newcommand{\nn}{\nonumber}
\newcommand{\ba}{\begin{array}}
\newcommand{\ea}{\end{array}}

\newcommand{\R}{{\mathbb R}}
\newcommand{\C}{{\mathbb C}}
\newcommand{\Z}{{\mathbb Z}}
\newcommand{\T}{{\mathbb T}}
\newcommand{\N}{{\mathbb N}}
\newcommand{\cL}{\mathcal{L}}
\newcommand{\cC}{\mathcal{C}}
\newcommand{\cI}{\mathcal{I}}
\newcommand{\K}{\mathrm{K}}  

\newcommand{\diag}{\mathrm{diag\,}}

\newcommand{\trace}{\mathrm{trace\,}}

\newcommand{\E}{\mathbb{E}}
\newcommand{\PP}{\mathbb{P}}

\newcommand{\iv}{^{-1}}
\newcommand{\iy}{\infty}
\renewcommand{\kappa}{\varkappa}
\newcommand{\eps}{{\varepsilon}}

\newcommand{\De}{\Delta}

\renewcommand{\rho}{\varrho}
\newcommand{\ta}{\tilde{a}}
\newcommand{\tb}{\tilde{b}}

\renewcommand{\a}{a_{\lambda}}

\newcommand{\m}{\mathfrak{m}}
\renewcommand{\uparrow}{\to}

\begin{document}

\title{Perturbed Toeplitz operators and radial determinantal processes}
\author{Torsten Ehrhardt \thanks{ehrhardt@math.ucsc.edu.}\\
 Department of Mathematics\\
 University of California\\
 Santa Cruz, CA 95064, USA
\and 
Brian Rider\thanks{brian.rider@colorado.edu.}\\
Department of Mathematics\\
University of Colorado at Boulder\\
Boulder, CO 80309, USA
}
\maketitle

\begin{abstract}
We study a class of rotation invariant determinantal ensembles in the complex plane;
examples include the eigenvalues of Gaussian random matrices and the  roots of
certain families of random polynomials.    The main result is a criteria for a central limit theorem to hold for 
angular statistics of the points. The proof  exploits an exact formula 
relating the generating function of such statistics to the determinant of a perturbed Toeplitz matrix.

\end{abstract}

\section{Introduction}
\label{s1}

Consider the probability measure on $n$ complex points, $z_1, \dots, z_n \in \C$, defined by
\be
\PP_{\m,n} (z_1,\dots z_n) = \frac{1}{Z_{\m,n}} \prod_{j<k} |z_j-z_k|^2 \prod_{k=1}^n d{\m}(z_k),
\label{density}
\ee
with a (positive) reference measure $\m$ on $\C$.
This is an instance of a {\em determinantal ensemble}, so named as the presence of the Vandermonde
interaction term $\prod |z_i -  z_j|^2$ results in all $k$-fold ($k \le n$) correlations of the points being 
given by a determinant of a certain $k \times k$ Gramian.    Determinantal ensembles as such were identified
in the mathematical physics literature as a model of fermions  \cite{Macchi}, but also arise
naturally in a number of contexts including random matrix theory.  For background,  \cite{BKPV} and
\cite{Sosh00} are recommended.  

Throughout the paper we restrict to the situation of radially symmetric weights,  $d\m(z) = d \mu (r)  d \theta$ ($ z = r e^{i \theta}$), 
also assuming that $\m$ has no unit mass at the origin.   The standard
examples in this set-up are the following:

\vspace{.3cm}
\noindent
{\it{Ginibre ensemble.}} Let $M$ be an $n \times n$ random matrix in which each entry is an independent 
complex Gaussian of mean zero and mean-square one.  Then the $n$ eigenvalues have joint density (\ref{density})
with $d\mu(r) = r e^{-r^2} dr$ \cite{Ginibre}.

\vspace{.3cm}
\noindent
{\it{Circular Unitary Ensemble (CUE)}.}   Place Haar measure on $n$-dimensional unitary group $U(n)$ and consider again the eigenvalues.   
These points live on the unit circle $\T = \{ t \in \C: |t| = 1\}$, and it is well known that their joint law is given by (\ref{density}) in which $\mu$ 
is the point mass at  one.

\vspace{.3cm}
\noindent
{\it{Truncated Bergman process.}}  Start with the random polynomial $z^n + \sum_{k=0}^{n-1} a_k z^k$ 
with independent coefficients drawn uniformly from the disk of radius $r$ in $\C$.  Condition the roots
$z_1, \dots, z_n$ to lie in the unit disk. Then, the $r \rightarrow \infty$ limit of the conditional root ensemble
is (\ref{density}) where now $\mu$ is the uniform measure on the disk of radius one. This nice fact may be 
found in \cite{Hammersley}; for an explanation of the name see \cite{PeresVirag}.

\vspace{.3cm}

Our aim is to identify criteria on $\mu$ under which a central limit theorem (CLT)
for the quantity 
$$
   X_{f,n} = \sum_{k=1}^n f(\arg z_k)
$$
holds or not.  Whatever criteria will depend on the regularity of the test function $f$ as well.  
An enormous industry has grown up around CLT's for linear statistics in determinantal and random matrix ensembles. Despite rather than
because of this, there are several reasons for making a special study of such ``angular" statistics in the given setting.

The conventional wisdom is that choosing $f$  sufficiently smooth produces Gaussian fluctuations
with order one variance ({\em i.e.}, as $n \uparrow \infty$ the un-normalized $X_{f,n} - \E X_{f,n}$ should posses a CLT).
This is borne out by a number of results pertaining to ensembles with symmetry and so real, or suitably ``one-dimensional",
spectra.  In the present context in which points inhabit the complex plane,  \cite{RV1} proves a result of this type for $C^1$ statistics of the Ginibre ensemble.
On the other hand, a smooth function of $\arg z$ is not smooth when regarded as a  function of the variable $z \in \C$. 
In fact, again for the Ginibre ensemble and for $f$ possessing an $L^2$-derivative,  \cite{Rid04} shows the variance of $X_{f,n}$
to be of order $\log n$ but is unable to establish a CLT.   While there are a number of general results on
CLT's for determinantal processes in whatever dimension,  notably \cite{Sosh02} which employs cumulants, 
the logarithmic growth in this case is not sufficiently fast for those conclusions to be relevant.  
We also mention that for any determinantal process on $\C$ with radially symmetric weight, the collections of moduli
$|z_1|, |z_2|, \dots $ are independent; this is spelled out nicely in \cite{BKPV}.  Hence, CLT's for ``radial'' statistics 
in our ensembles may be proved via the classical Lindenberg-Feller criteria, see \cite{For99} and \cite{Rid04}
for details in the Ginibre case.

It is likely that the considerations of \cite{RV1}, which entail a refinement of the cumulant method,  can be 
adopted to the matter at hand.  Here though we take an operator-theoretic approach, based on the following formula.
For any $\varphi \in L^{\infty}(\T)$,
\begin{equation} 
\label{detform}
 \E_{\m,n} \left[  \prod_{k=1}^n \varphi( \arg z_k )  \right]  =  \det  M_{\mu, n} (\varphi), \qquad
 M_{\mu, n} (\varphi) = 
  ({\varphi}_{k-\ell} \, \rho_{k,\ell} )_{0 \le k, \ell \le n-1},  
\end{equation}
where ${\varphi}_k = \frac{1}{2 \pi} \int_{0}^{2\pi} \varphi(x) e^{ik x} dx$, the $k$-th Fourier coefficient of $\varphi$, and 
\begin{equation}
\label{rho}
    \rho_{k,\ell} = \frac{ m_{k+\ell}}{ (m_{2k} \, m_{2 \ell})^{1/2} }  \   \   \mbox{ in which } \   \    m_k = \int_0^{\infty} r^{k} \,d\mu(r),  
\end{equation}
the $k$-th moment  of the half-line measure $\mu$.   The brief derivation of (\ref{detform}) can be found in the appendix.

This provides  an explicit formula for the generating 
function of $X_{f,n}$ by the choice $\varphi = e^{i \lambda f}$.  A CLT for $X_{f,n}$ will then follow from sufficiently
sharp $n \uparrow \infty$ asymptotics of the determinant on the right hand side of (\ref{detform}). Of  course, if this is to be the strategy
we must henceforth assume that $m_k < \infty$ for all $k$.

In the case of CUE, all $m_k=1$, and the identity (\ref{detform}) reduces to  
Weyl's formula relating the Haar average of a class function in $U(n)$ to a standard Toeplitz determinant. 
The strong Szeg\"o limit theorem and its generalizations to symbols of weaker regularity then 
imply a variety of CLT's for linear spectral statistics in $U(n)$, see for instance \cite{HKO} and references therein.  
For more generic $\mu$, what appears on the right hand side of (\ref{detform}) is  the Hadamard product
of (truncated)  Toeplitz and Hankel operators.  While Hankel determinants arise as naturally as their Toeplitz counterparts 
in random matrix theory and several applications have prompted investigations of Toeplitz $+$ Hankel forms 
(see for example \cite{BasorEhr}), the present problem is the first to our knowledge to motivate an asymptotic
study
of Toeplitz $\circ$ Hankel matrices.  Though, as the title suggests, the analysis more closely follows the 
Toeplitz framework.

To describe the regularity assumed on the various test functions $f$, we introduce the function space
$F \ell^p(\nu)$, $1\le p<\iy$ (see \cite{Ka}), comprised of all $f \in L^1(\T)$ such that 
\be\label{Fl}
\|f\|_{{F\ell}^{p}(\nu)}:=\left(\sum_{n=-\iy}^\iy |f_n|^p \nu_n \right)^{1/p}<\iy.
\ee
Here $\nu=\{\nu_n\}_{n=-\iy}^\iy$ is a positive weight. (As above, $f_n$ stands for the Fourier coefficients of $f$.)
We will in particular deal with the cases $p=1$ or $p=2$, 
and power weights $\nu_n=(1+|n|)^\sigma$, $\sigma\ge0$.
In the latter case we simply denote the space by $F\ell^p_{\sigma}$ and write $F\ell^p$ when $\sigma=0$.

As for the underlying probability measure $\mu$,  
a natural criteria arises on the second derivative of the logarithmic moment function.

\medskip\medskip\noindent
{\bf Moment assumption.}
The function 
\be\label{def.m-xi}
 \xi \mapsto m_\xi := \int_0^{\infty} r^{\xi}\, d \mu(r),\qquad\xi\ge0,
\ee
satisfies one of the following two sets of conditions.
\medskip

\noindent
\underline{{\bf (C1)}  or ``$\beta > 1$"}: 
It holds
\be
(\ln m_{\xi})'' = O(\xi^{-\beta}),\qquad  \xi\to \iy,
\ee
with $\beta > 1$.

\medskip

\noindent
\underline{{\bf (C2)} or  ``$1/2  < \beta \le 1$"}: It holds 
\be
(\ln m_\xi)'' =h_\mu(\xi) + O(\xi^{-\rho}), \qquad \xi\to \iy,
\ee
for a differentiable function $h_\mu(\xi)\ge0$, $\xi>0$, such that 
\be
h_\mu(\xi)=O(\xi^{-\beta}),  \qquad h'(\xi)=O(\xi^{-\gamma}),\qquad \xi\to\iy,
\ee
with $1/2<\beta\le 1$, $\rho, \gamma>1$. 
Additionally, 
\be\label{def.iota}
\iota_\mu(x):=\frac{1}{2}\int_1^x h_\mu(\xi)\, d\xi,
\ee
tends to infinity as $x\to\iy$.
\qed
\medskip
\medskip

Notice that since we have already assumed $m_k < \infty$ for all $k$, $m_\xi$ is infinitely differentiable for positive $\xi$. 
The typical behavior we have in mind in both (C1) and (C2) are asymptotics like
\be\label{cond.typ}
(\ln m_\xi)'' =\alpha \xi^{-\beta}+O(\xi^{-\rho}),\qquad \xi\to\iy,
\ee
with $\alpha,\beta>0$ and $\rho>1$ .
As examples, we remark that for Ginibre, $(\ln m_\xi)'' = \frac{1}{2} {\xi^{-1}} + O(\xi^{-2})$, while both CUE and truncated Bergman satisfy $(\ln  m_\xi)'' =  O(\xi^{-2})$.  The transition from $\beta\le 1$ to $\beta > 1$ is particularly interesting; 
Section 2 discusses the moment conditions in greater detail.
The restriction to $\beta > 1/2$ is tied to the method in which we show that $M_{\mu,n}$ 
is a small perturbation of the associated Toeplitz
form, in either trace or Hilbert-Schmidt norm, and this breaks down at $\beta = 1/2$.  By considering the perturbation in
higher Schatten norms  it may be possible to push our strategy further.

\begin{theorem} 
\label{thm1}  
Assume the moment condition {\em (C2)}, and let $\sigma=\max\{1/\beta,3/(2 \gamma)\}$. 
Then, for real-valued $f \in F \ell_{\sigma}^2$, the normalized statistics
$$ 
X^{\mathrm{scal}}_{f,n}:=\frac{X_{f,n} - n f_0 }{ \sqrt{ \iota_\mu(2n) ñ }} 
$$
converges in law to a mean zero Gaussian with variance $\sum\limits_{k \in \Z } k^2 |f_k|^2$
as $n \to \infty$.
\end{theorem}

If we assume the particular asymptotics (\ref{cond.typ}), then we obtain
$$
\iota_\mu(2n)=\left\{\begin{array}{cc} \frac{\alpha\log(2n)}{2} & \beta=1,\\[.5ex]
\frac{\alpha(2n)^{1-\beta}}{2(1-\beta)} & 1/2<\beta<1,
\end{array}\right.
$$
which up to the constant stated in the theorem is the asymptotics of the variance of $X_{f,n}$.
For canonical $\beta=1$ cases like Ginibre, we have $\sigma=1$  and hence the assumed regularity on $f$ is optimal.
For $\beta < 1$, because the asymptotic variance of $X_{n,f}$ is $\sim n^{1-\beta}$ and the mean is $\sim n$, 
one may conclude a CLT  from  \cite{Sosh02} (even for $\beta \le 1/2$), though for  possibly different classes of $f$.
This  highlights what our method can and cannot accomplish.

Next we define the infinite version of the matrix $M_{\mu,n}$ and
the related Toeplitz operator $T$,
\be\label{def.MT}
M_{\mu}(a)= \left(\rho_{j,k}a_{j-k}\right), \qquad 
T(a)= \left(a_{j-k}\right),\qquad
j,k\ge 0,
\ee
both viewed as bounded linear operators on $\ell^2=\ell^2(\Z_+)$, 
$\Z_+=\{0,1,2,\dots\}$.  

\begin{theorem} 
\label{thm2}
Assume the moment condition {\em (C1)}, and assume $f$ to be real-valued.
\begin{enumerate}
\item[(a)] 
If $f \in F \ell_{ {1}/{\beta}}^2$ for $\beta < 2$ or $f \in F \ell_{{1}/{2} }^2 \cap L^{\infty}(\T)$
for $\beta \ge 2$, then
$$
     X_{f,n} - n f_0   \Rightarrow  \mathcal{Z}
$$
as $n\to\iy$ with a mean-zero random variable  $\mathcal{Z} = \mathcal{Z}(f;\mu)$.
\item[(b)]
If  $f\in F\ell^1_\sigma$ or $f \in F \ell^2_{\sigma+\varepsilon}$, where $\sigma=\max\{1,2/\beta\}$, $\eps>0$, then the cumulants $c_m$ of $\mathcal{Z}$ may be described as follows.
Introduce the recursion
$$
   C_m =  M_\mu(f^m) - \sum_{k=1}^{m-1} {m-1 \choose k} C_{m-k} M_\mu(f^{k}), \quad m \ge 1.
$$
Then $c_2(\mathcal{Z}) =  Var( \mathcal{Z}) =  \trace C_2 +  \sum_{k=1}^\iy k |f_k|^2 $,
while   $c_m(\mathcal{Z}) = \trace C_m$ for $m \ge 3$.
\end{enumerate}
\end{theorem}

For  CUE, $ \rho_{k, \ell} \equiv 1$ and one can check that $c_2(\mathcal{Z}) = 2 \sum_{k=1}^\iy k |f_k|^2$,
$c_m = 0$ for all $m \ge  3$ and so $\mathcal{Z}$ is Gaussian.
That is to say the obvious: Theorem \ref{thm2} reduces to the strong Szeg\"o theorem. 
 In general though it does not appear efficient to compute the
cumulants of $\mathcal{Z}$ from the formula above, even in explicit, and seemingly simple examples like
truncated Bergman for which $\rho_{k, \ell} = \frac{ 2 \sqrt{(k+1)(\ell +1)}}{ k +\ell +2}$. 
 The more basic
problem which remains open is to determine when 
$\mathcal{Z}$ is Gaussian, {\em i.e.}, for what weights $\mu$ does  $c_m$ vanish for all $m\ge 3$.
We conjecture this is only the case for CUE, when $\mu$ is a unit mass.
The intuition is that whenever say $\mu$ is compactly supported, the 
normalized counting measure of points concentrates on the boundary of a disk 
(as in CUE, this is discussed further in Section 2).  If however $\mu$ has extent (is not concentrated at one place),
 there  remains   a positive number of points of modulus $< 1$ with probability one as $n \uparrow \infty$;  their non-normal law will not wash in the 
type of centered (but not scaled) limit considered in Theorem \ref{thm2}.

Theorems \ref{thm1} and \ref{thm2} are 
intimately connected to the following, direct generalization of
the Szeg\"o-Widom Limit Theorem to the determinants of $M_{\mu,n}(a)$.

\begin{theorem}\label{thm3}\
\begin{enumerate}
\item[(a)]
Assume the moment condition {\em (C2)}, let $\sigma=\max\{1/\beta,3/(2\gamma)\}$ and
 $B=F\ell^2(\nu)$ such that $\nu_m=\nu_{-m}$, $\nu_m$ is increasing ($m\ge 1$), and
\be\label{cond.nu}
\nu_m\ge \max\left\{(1+|m|)^\sigma,\sqrt{1+m^2\iota_\mu(2|m|^{2\sigma})}\right\},
\qquad \sup\limits_{m\ge 1}\frac{\nu_{2m}}{\nu_m}<\iy.
\ee
Let $a\in B$ and suppose $T(a)$ is invertible on $\ell^2$.
Then
\be\label{con.1}
\lim_{n\to\iy}
\frac{\det M_{\mu,n}(a)}{G[a]^n\exp(\iota_\mu(2n)\Omega[a])}=F[a],
\ee
with some constant $F[a]$  and
\be\label{Ga}
G[a]=\exp([\log a]_0),\qquad
\Omega[a]=\frac{1}{2}\sum_{k=-\iy}^\iy k^2 [\log a]_k[\log a]_{-k}.
\ee
\item[(b)]
Assume the moment condition {\em (C1)}, let
$a\in L^\iy(\T)\cap F\ell^2_{1/2}$ if $\beta\ge2$  or $a\in F\ell^2_{1/\beta}$ if $1<\beta<2$.
Suppose $T(a)$ is invertible on $\ell^2$. Then
\be\label{con.2}
\lim_{n\to\iy}
\frac{\det M_{\mu,n}(a)}{G[a]^n}=E[a],
\ee
for a constant $E(a)$.   
If further  $a\in F\ell^1_\sigma$ or $a\in F\ell^2_{\sigma+\eps}$, $\sigma=\max\{1,2/\beta\}$, $\eps>0$, there
is the expression
\be
E[a]=\det\Big( T(a\iv) M_\mu(a) \Big).
\ee
\end{enumerate}
The convergences in (\ref{con.1}) and (\ref{con.2}) is uniform in $a$ on compact subsets
of the function spaces.
\end{theorem}

The assumption that $T(a)$ is invertible is a natural assumption on
the symbol; it is the condition in the (scalar) Szeg\"o-Widom  theorem  (see \cite[Ch.\ 10]{BS}, and \cite{Wi}). 
One of the general versions of that theorem pertains to symbols drawn from 
the Krein algebra $\K=L^\iy(\T)\cap F\ell^2_{1/2}$ (which contains discontinuous functions).
Hence, at least for $\beta \ge 2$, we achieve  the same level of generality.

Except for the Krein algebra $\K$, the various classes of symbols occurring above are Banach algebras 
continuously embedded in $C(\T)$. 
For those classes, the assumption on $a$ is equivalent to requiring that $a$ possesses a continuous logarithm on $\T$, which then enters the definition of the constant $G[a]$ and $\Omega[a]$. In case of  $\K$, we must define
\be\label{defGalt}
G[a]=[T\iv(a\iv)]_{00}
\ee
as the $(0,0)$-entry in the matrix representation of the inverse Toeplitz operator, as is well known
in the context of the classical Szeg\"o-Widom theorem.

The quite technical assumptions in (\ref{cond.nu}) can be simplified in special situations such as (\ref{cond.typ}).
Then $\iota_\mu(x)=\frac{\alpha x^{1-\beta}}{2(1-\beta)}$ ($1/2<\beta<1$) or  $\iota_\mu(x)=\frac{\alpha\log x}{2}$ ($\beta=1$). Consequently,
in case $1/2<\beta<1$ we can take $B=F\ell^2_\sigma$, while in case $\beta=1$ we can take $B=F\ell^2(\nu)$,
 $\nu_m=C(1+|m|)\log^{1/2}(2+|m|)$, which is only slightly stronger than one might expect.

The theorems above are derived in  Sections 6 and 7, as a consequence of a more general result, Theorem \ref{mainthm} 
(Section 4), on the asymptotics of determinants of type (\ref{detform}).    Section 3 lays out various preliminaries required for the proof
of Theorem \ref{mainthm}, and also explains how we employ the moment assumption. Section 5 provides 
detailed asymptotics of a certain trace term occurring in Theorem \ref{mainthm} which
is tied to the variance of $X_{f,n}$.

We close the introduction by pointing out that since we focus on angular statistics, it is the same to 
consider fixed reference measures $d \mu(r)$ as it is $n$-dependent measures of the form $ d \mu_n(r) = d \mu(c_n r)$
for some scale factor $c_n$.  There are though examples of interest which fall out of this set-up.  For instance,
there is the spherical ensemble connected  to $A^{-1} B$ in which $A$ and $B$ are independent
$n\times n$ Ginibre matrices.  The resulting eigenvalues form a determinantal process with $d\mu_n(r)  = r (1+r^2)^{-(n+1)} dr$
\cite{Manju}.  Another example are the roots of the degree-$n$ complex polynomial with Mahler measure one, 
for which $d \mu_n(r) =  r \min(1, r^{-2n-2}) dr $ \cite{Chen}.    Our methods could perhaps be adopted to both situations, but we do not pursue this.

\section{On the moment condition}
\label{s2}

Of the key examples, both  CUE and 
truncated Bergman satisfy $  ( \ln m_{\xi})'' =  O(\xi^{-2})$, while
the Ginibre ensemble satisfies
$ ( \ln m_{\xi})'' = O(\xi^{-1})$.
A few more illustrative examples are contained in the following.

\begin{proposition}
\label{prop2}
Consider positive measures on $\R_+$ with density $d \mu(r) = \mu(r) dr$ and corresponding moment function
$m_{\xi} = \int_0^{\infty} r^{\xi} \mu(r) dr$.  
\begin{itemize}
\item[(i)] 
If $\mu(r)$ is supported on a finite interval  $[a,b]$, and is ``regular" at $b$ as in $\mu(r) = c (b -r)^{\alpha-1}$ for $r \in (b-\delta, b]$
and $\alpha > 0$, then $(\ln m_{\xi})'' =  {\alpha}{\xi^{-2}} +O(\xi^{-3}) $.
\item[(ii)] 
If $\mu(r) = p(r) e^{- c r^{\alpha}} $ for polynomials $p$ and  $\alpha >0$, then    $(\ln m_{\xi})'' =  \alpha \xi^{-1} + O(\xi^{-2})$.
\item[(iii)]
 If $\mu(r) = e^{- c (\ln (e+r))^{q}} $ for $q > 1$, then $(\ln m_{\xi})'' =  \alpha \xi^{\frac{2-q}{q-1}}  + O(\xi^{\frac{3-2q}{q-1}} )$
upon choosing $ c = \alpha^{1-q} ( q^{1/(1-q)} - q^{q/(1-q)})$.  
\end{itemize}
\end{proposition}

\begin{proof}  We start with explicit instances of cases (i) and (ii).   For (i), there is no loss in assuming that $[a,b]=[0,1]$ and
we consider further $\mu^{(i)}(r) = (1-r)^{\alpha-1} 1_{[0,1]}$.  For case (ii), consider a simple polynomial term $\mu^{(ii)}(r) = r^p e^{-r^{\alpha}}$.
Then we have, 
$$
\ln {m}_\xi^{(i)} 
   =  \ln\Gamma(\xi+1)  -  \ln \Gamma(\xi+\alpha+1) + \ln  \Gamma(\alpha),
$$
and 
$$
  \ln m_\xi^{(ii)} = \ln \Gamma((\xi +  p+1)/\alpha)   - \ln \alpha.
$$ 
From this point the verifications may be completed by use of the appraisal
$\frac{d^2}{dz^2} \ln \Gamma(z)  $$=$$  {z^{-1}} + (1/2) z^{-2} + O( z^{-3}),$
valid for large real values of $z$.

More generally,  for case (i) we write
$$
     (\ln m_{\xi})'' = \frac{ \langle (\ln r)^2 r^{\xi} \rangle_{\mu}}{  \langle  r^{\xi} \rangle_{\mu}   }  - \frac{ \langle (\ln r) r^{\xi} \rangle_{\mu}^2}{  \langle  r^{\xi} \rangle_{\mu}^2   }
$$
and note that Laplace asymptotic considerations yield: for $d=0,1,2$,  
$ \langle (\log r)^d r^{\xi} \rangle_{\mu} =  \langle (\log r)^d r^{\xi}  \rangle_{\mu^{(i)}} $ $+$ $ O(e^{-C_\delta \xi}),$
which is more than enough to show that one has the same asymptotics for any such $\mu$ as for ${\mu}^{(i)}$.
That (ii) extends to more general polynomials $p(r)$ is self-evident.

For case (iii) we only mention that it is most convenient to consider the asymptotically equivalent object
$  {m}_{\xi} = \int_0^{\infty}  e^{\xi r - c r^q} dr $ (after an obvious change of variable) for which the 
leading order arises from a neighborhood of the stationary point  $r^* = (\xi/c q)^{ \frac{1}{q-1}}$. 
The details are
straightforward.
\end{proof}

The above is intended to be illustrative;  no attempt to optimize the regularity conditions 
on $\mu$ has been made. We also mention here without proof that the measure  $d\mu(r) = e^{-e^r} dr$
produces a moment sequence for which there is the not strictly polynomial decay 
$ ( \ln m_{\xi})'' =  O( \frac{1}{ \xi \ln \xi})$.  Further, by Fourier inversion, 
one may produce measures for which $(\log m_{\xi})''$ is exactly $\alpha(1+\xi)^{-\beta}$ for $0< \beta \le 2, \beta \neq 1$, $\alpha>0$.

\paragraph{Moment condition and the mean measure}

Our  condition(s) on the moment sequence also dictate the limit shape of the 
 mean measure of the points.    This object is
given by
$$
   d {\Lambda}_n (z)  =   \left( \frac{1}{n} \sum_{k=1}^{n-1} \frac{|z|^{2k}}{2 \pi  m_{2k} }  \right)     d^2z,  
$$
where  $d^2z$ denotes Lebesgue measure on $\C$, and as the name suggests $ \E_{\m,n} [ \# \mbox{ points in } A ] = 
n \int_A d \Lambda_n(z)$ for (measurable) $A \subseteq \C$, see again \cite{BKPV}.  We provide one description of  
the shift from a ``$\beta=1$" setting, resulting in an extended limit support, to a ``$\beta>1"$ setting for which the limit support is
degenerate.  This is in line with the conjecture discussed after Theorem \ref{thm2}.

\begin{proposition} 
\label{prop1}
For all sufficiently large $\xi$ let the moment sequence $ m_\xi = \int_0^{\infty} r^{\xi}  d\mu(r)$ satisfy
\be
\label{momentcond}  
   ( \ln m_\xi )^{\prime \prime} = \frac{\alpha}{\xi+1} +   \eps(\xi)
\ee
with  $\alpha \ge 0$ and $\eps \in L^1(\R_+)$.
Then there exists
a rescaling of $\PP_{\m,n}$  so that $d \Lambda_n$ 
converges weakly to either: a weighted circular law with 
density $ \frac{1}{2 \pi \alpha} |z|^{\frac{1}{\alpha} -2}$ on $|z| \le 1$ when $\alpha>0$, 
or  to the uniform measure on $|z|=1$ when $\alpha=0$.
\end{proposition}
 
Note,  $\eps$ is necessarily nonnegative when $\alpha =0$.  And of course, when $\alpha = 1/2$ the advertised limit 
is the standard circular law (see e.g. \cite{Bai}).

\begin{proof} Choose $q \gg 1$ so that  (\ref{momentcond}) is in effect for $s \ge q$, and then integrate the equality twice:
first over $q \le s \le t$,  and then in $t$ from $k$ to $k+\ell$ to find
\bq
 \ln \left( \frac{m_{k+\ell}}{m_k}  \right)  & = & 
                                                                    \alpha \ell \log k   + c \ell + o(1).    \label{moment2}
\eq
(Here $c =    (\ln m)^{\prime}(q) + \alpha  \ln m(q) - \int_q^{\infty} \varepsilon(s) ds$, and the $o(1)$ holds in $k$  $-$ we
view $\ell$ as fixed).
Next compute the $\ell^{th}$ absolute moment in the mean measure:
\be
  \int_{\C} |z|^{\ell} d \Lambda_n(z) =  \frac{1}{n} \sum_{k=0}^{n-1}  \frac{m_{k+\ell}}{m_k} =
   \frac{1}{n} e^{c \ell}  \sum_{k=1}^{n-1} k^{\alpha \ell}(1+o(1)).
\label{meanmoments}
\ee
Neglecting the multiplicative errors, in the case $\alpha=0$  the sum (\ref{meanmoments}) converges  to $e^{c\ell}$ 
for any $\ell$, unambiguously the moment sequence defined by placing unit mass at the place $e^c \in \R_+$.
When $\alpha > 0$,  we rescale 
$\PP_{\m,n}$   by sending $\{ z_i\}_{ 1 \le i \le n} \mapsto \{ n^{-\alpha } z_i\}_{ 1 \le i \le n} $.
 Then, the sum becomes 
$     
       e^{c \ell}  \frac{1}{n} \sum_{k=1}^{n-1}  (k/n)^{\alpha \ell}   
       \rightarrow  \frac{ e^{c \ell} }{ \alpha \ell + 1}
$
as $n \rightarrow \infty$.  Matching constants in
$
     \int_0^b t^{\ell}  \, d  (t/b)^{p+1}   = \frac{p+1}{p+\ell +1} b^{\ell}
$
identifies (uniquely) the scaled $\alpha > 0$ moment sequence with that of the measure with density $f(t) = (p+1) t^p/ b^{p+1}$ on $[0, b]$
where $b = e^c$ and $p = \frac{1-\alpha}{\alpha}$.    Thus  the limit mean measure  (or actually its radial
projection) is also identified.
In either case, $\alpha > 0$ or $\alpha =0$, an additional rescaling will pull the edge of the support from $e^c$ to $1$.
\end{proof}


\section{Hilbert-Schmidt and trace class conditions}
\label{s3}

Our results  hinge on being able to consider $M_\mu(a)$ as a suitable compact perturbation of the Toeplitz operator $T(a)$ 
(see (\ref{def.MT})).  Here we will establish sufficient conditions on $a$ and $\mu$ such that  
$$
K_{\mu}(a) =M_\mu(a)-T(a) = \left( (\rho_{j,k}-1)a_{j-k})\right),  \qquad j,k\ge 0,
$$
is Hilbert-Schmidt or trace class operator. We refer to \cite{GoKr} for general information about these notions. 
Since $T(a)$ is bounded on $\ell^2$ whenever $a\in L^\iy(\T)$, under the appropriate conditions $M_{\mu}(a)$ is then also bounded.
While it might be interesting to ask for necessary and sufficient conditions for the boundedness of $M_{\mu}(a)$
and the compactness of $K_\mu(a)$, we think it is a non-trivial issue, which we will not pursue here.

The compactness properties of  $K_{\mu}(a)$ rely mainly on the ``shape'' of $\rho_{j,k}$ near the diagonal. 
An application of H\"older's inequality shows that  $0 < \rho_{j,k} \le 1$.  
More detailed information on $\rho_{j,k}$ is provided by the following technical lemma, for which we use 
the set of indices,
\be\label{I.delta}
\mathcal{I}_\delta =
\Big\{\, ( j,k)\in \Z_+\times\Z_+ : \, |j-k|^\delta < (j+k)/2\,\Big\},
\ee
always assuming $\delta\ge 1$ ($\Z_+=\{0,1,\dots\}$). The factor $1/2$ in $\cI_\delta$ is only for technical convenience.  In particular, $(j,k)\in\cI_\delta$ implies $j,k\ge1$.

Part (a) of the  lemma will be used at several places, while the more elaborate part (b) is used only in Lemma \ref{l5.2}.  Part (b) obviously implies part (a), but it seems more clarifying to state and prove (a) separately.  Throughout what follows we will utilize the notation $a\vee b:=\max\{a,b\}$.

\begin{lemma}\label{l3.1}\
\begin{itemize}
\item[(a)]
Let $\beta>0$, $\delta\ge1$, $\beta\delta\ge 2$, and assume that the measure $\mu$ satisfies the condition
$$
(\ln m_\xi)''  = O(\xi^{-\beta}),\qquad \xi\to\iy.
$$
Then,  for $(j,k)\in\cI_\delta$ with $\Delta=j-k, \sigma=j+k$, we have the uniform estimate
\bq\label{rho-asym-1}
\rho_{j,k} &=& 1+ O\left(\frac{\Delta^2}{\sigma^{\beta}}\right).
\eq

\item[(b)]
Let $\beta,\gamma,\rho>0$, $\delta\ge1$, $\beta\delta\ge 2$,  $\gamma\delta \ge 3$, $\rho\delta\ge2$, and
assume that there exists a differentiable function $h_\mu(\xi)\ge0$ such that
$$
(\ln m_\xi)''  =h_\mu(\xi)+O(\xi^{-\rho}),\qquad \xi\to\iy,
$$
and
$$
h_\mu(\xi) = O(\xi^{-\beta}),\qquad 
h'_\mu(\xi) = O(\xi^{-\gamma}), \qquad \xi\to\iy.
$$
Then, for $(j,k)\in\cI_\delta$ with $\Delta=j-k, \sigma=j+k$, we have the uniform estimate
\bq\label{rho-asym-2}
 \rho_{j,k} &=& 
1-\frac{\Delta^2}{2} h_\mu(\sigma)
+O\left(
\frac{\Delta^4}{\sigma^{2\beta}}\vee
\frac{|\Delta|^3}{\sigma^\gamma}\vee
\frac{\Delta^2}{\sigma^\rho}\right).
\eq
\end{itemize}
\end{lemma}
\begin{proof}
We can  assume without loss of generality that $\Delta>0$. Then 
$$
\ln \rho_{j,k} = \ln m_{\sigma}-\frac{\ln m_{\sigma+\De}+\ln m_{\sigma-\De}}{2}  
=- \frac{\De^2}{2} (\ln m_\eta)'', 
 \qquad \eta\in (\sigma-\De,\sigma+\De), \nn 
$$
after applying the mean-value theorem twice.
We can write $\eta=\sigma(1+\tau)$, where the error term $\tau$ is estimated by $|\tau|\le |\De|/\sigma\le |\Delta|^\delta/\sigma\le 1/2$ using $\delta\ge1$.

In case (a) we can conclude that
$$
\ln \rho_{j,k}=O\left(\frac{\Delta^2}{\eta^\beta}\right)=O\left(\frac{\Delta^2}{\sigma^\beta}\right).
$$
Because $\beta\delta\ge 2$ we get $\Delta^2\le \sigma^{2/\delta}\le \sigma^\beta$.
Hence the above term is bounded and exponentiating yields the assertion.
In case (b) we first obtain
$$
\ln \rho_{j,k}=-\frac{\Delta^2}{2}h_\mu(\eta)+O\left(\frac{\Delta^2}{\eta^\rho}\right).
$$
Now we apply once more the mean value theorem to obtain the estimate
$$
\ln \rho_{j,k}=-\frac{\Delta^2}{2} h_\mu(\sigma) +O\left(\frac{|\Delta|^3}{\sigma^\gamma}\vee
\frac{\Delta^2}{\sigma^\rho}\right).
$$
Notice that, as above, $\eta=\sigma(1+\tau)$ with $|\tau|\le 1/2$. 
All these terms are bounded because $2/\delta\le \beta$, $3/\delta\le\gamma$, and $2/\delta\le \rho$. 
The assertion is obtained upon exponentiating.
\end{proof}


Part (a) of the lemma translates immediately into the estimates that follow. 

\begin{proposition}\label{p3.2}
Let $\beta>1/2$ and assume that the measure $\mu$ satisfies the assumption
$$
(\ln m_\xi)''=O(\xi^{-\beta}),\qquad \xi\to\iy.
$$
Put $\sigma=1/2\vee 1/\beta$. Then there exists a constant $C_\mu>0$ such that 
$K_\mu(a)$ is Hilbert-Schmidt and the estimate
$$
\| K_\mu(a)\|_{\cC_2(\ell^2)} \le C_{\mu} \|a\|_{F\ell^2_\sigma}
$$
holds whenever $a\in F\ell^2_\sigma$.
\end{proposition}

\begin{proof}
Put $\delta=2\sigma=1\vee 2/\beta$ so that Lemma \ref{l3.1}(a) is applicable.
The operator $K_\mu(a)$ is Hilbert-Schmidt if and only if the sum 
$
\sum_{(j,k)\in\Z_+^2} |a_{j-k}|^2 (1-\rho_{j,k})^2
$
is finite (this quantity is the square of the Hilbert-Schmidt norm). 
We have that
\bq 
\sum_{(j,k)\in\Z_+^2} |a_{j-k}|^2 (1-\rho_{j,k})^2 & \le &
\sum_{(j,k)\notin\cI_\delta} |a_{j-k}|^2+\sum_{(j,k)\in\cI_\delta} |a_{j-k}|^2(1-\rho_{j,k})^2 \nn \\
& \le & \sum_{(d,s) \in \Z\times\Z_+ \atop |d|^\delta\ge  s/2}|a_d|^2 +
\sum_{(d,s) \in \Z\times\Z_+ \atop |d|^\delta <  s/2}|a_d|^2\frac{d^4}{s^{2\beta}}  \nn \\
& \le & C \sum_{d\in \Z} |a_d|^2 |d|^{\delta} + C
\sum_{d\in \Z} |a_d|^2 |d|^{4+\delta(1-2\beta)}. \nn
\eq
Line one just uses $\rho_{j,k}\in(0,1]$. In line two we make the substitution $d=j-k$, $s=j+k$
and employ Lemma \ref{l3.1}(a), and the final line uses the fact $\beta > 1/2$.
Furthermore, as $\delta\beta\ge2$ we see that the second term in this last line does not exceed the first one, 
and that in turn is equal to the square of $\|a\|_{F\ell_{\sigma}^2}$ ($\delta=2\sigma$).
\end{proof}

Next we  establish two sufficient conditions for $K_\mu(a)$ to be trace class.
It is not hard to show 
that one is not weaker than the other, {\em i.e.}, neither of the  two function classes 
pointed out below is  contained in the other.

\begin{proposition}\label{p3.3}
Let $\beta>1$ and assume that the measure $\mu$ satisfies the assumption
$$
(\ln m_\xi)''=O(\xi^{-\beta}),\qquad \xi\to\iy.
$$
Put $\sigma=1\vee 2/\beta$. Then there exists $C_\mu>0$ and, for each $\eps>0$, $C_{\mu,\eps}>0$
such that 
\begin{itemize}
\item[(a)] 
$K_\mu(a)$ is trace class and the estimate
$$
\| K_\mu(a)\|_{\cC_1(\ell^2)} \le C_{\mu} \|a\|_{F\ell^1_\sigma}
$$
holds whenever $a\in F\ell^1_\sigma$;
\item[(b)]
$K_\mu(a)$ is trace class and the estimate
$$
\| K_\mu(a)\|_{\cC_1(\ell^2)} \le C_{\mu,\eps} \|a\|_{F\ell^2_{\sigma+\eps}}
$$
holds whenever $a\in F\ell^2_{\sigma+\eps}$.
\end{itemize}
\end{proposition}

\begin{proof}  
Here we put $\delta=\sigma=1\vee 2/\beta$ and notice that then Lemma \ref{l3.1}(a) is again applicable.

(a):\ 
We first estimate the trace norm of $K_\mu(t^m)$, $m\in\Z$. Without loss of generality assume
$m>0$. Then $K_\mu(t^m)$ has entries on the $m$-th diagonal given by
$\{\rho_{k+m,k}-1\}_{k=0}^\infty$. This operator is trace class if and only if its trace norm 
$$
\sum_{k=0}^\iy |\rho_{k+m,k}-1|<\iy.
$$
We split and overestimate this sum by a constant times
$$
\sum_{(k+m,k)\notin \cI_\delta} 1 +\sum_{(k+m,k)\in \cI_{\delta}} \frac{m^2}{(2k+m)^\beta},
$$
using Lemma \ref{l3.1}(a) for the second part.
Now $(k+m,k)\in \cI_\delta$ means that $m^\delta < (2k+m)/2$, {\em i.e.}, $2k>2m^\delta-m$.
Noting that $2k\le 2m^\delta-m$ implies $k< m^\delta$, and $2k>2m^\delta-m$ implies $2k>m^\delta$,  the previous terms are overestimated by
$$
\sum_{0\le k < m^{\delta}} 1+\sum_{k\ge m^{\delta}/2} \frac{m^2}{(2k)^\beta}
\le m^\delta+C m^{2+\delta(1-\beta)}\le (1+C) m^{\delta}.
$$
Here we used $\beta>1$ and $\delta\beta\ge 2$, and all estimates are uniform in $m$. Thus $\|K_\mu(t^m)\|_{\cC_1(\ell^2)}=O(|m|^\delta)$. From here the proof of (a) follows immediately.

(b):\
Introduce the diagonal operator $\Lambda=\diag((1+k)^{-1/2-\eps})$, $\eps>0$, acting on $\ell^2$. As $\Lambda$ is Hilbert-Schmidt, and it suffices to prove that the operator with the matrix representation of $K_\mu(a)\Lambda\iv$ is Hilbert-Schmidt.
The squared Hilbert-Schmidt norm of $K_\mu(a)\Lambda\iv$ equals
$$
\sum_{(j,k)\in\Z_+^2} |a_{j-k}|^2 (1+k)^{1+2\eps} (1-\rho_{j,k})^2.
$$
As before we split the sum into two parts, 
$$
\sum_{(j,k)\notin\cI_\delta} |a_{j-k}|^2 (1+j+k)^{1+2\eps}
+\sum_{(j,k)\in\cI_\delta} |a_{j-k}|^2(1-\rho_{j,k})^2(1+j+k)^{1+2\eps},
$$
slightly overestimating it further. Now we make the substitution  $d=j-k\in\Z$ and $s=j+k\in\Z_+$.
We arrive at the upper estimate for the first term 
$$
\sum_{(d,s) \in \Z\times\Z_+ \atop |d|^\delta\ge  s/2}|a_d|^2 (1+s)^{1+2\eps} \le C
\sum_{d\in \Z} |a_d|^2 (1+|d|)^{\delta(2+2\eps)}
\le C \|a\|_{F\ell^2_{\delta(1+\eps)}}^2.
$$
For the second term, employ $(1-\rho_{j,k})^2 \le C\,(j-k)^4(1+j+k)^{-2\beta}$, by Lemma \ref{l3.1}(a), to find that it is bounded by a constant times
$$
\sum_{(j,k)\in\cI_\delta} |a_{j-k}|^2 \frac{(j-k)^4}{(1+j+k)^{2\beta-1-2\eps}}\le 
\sum_{(d,s)\in\Z\times\Z_+ \atop |d|^\delta < s/2} 
|a_{d}|^2 \frac{d^4}{(1+s)^{2\beta-1-2\eps}}.
$$
Without loss of generality we could have chosen $\eps>0$ small enough such that $\beta>1+\eps$.
Then we can estimate further by a constant times
$$
\sum_{d\in\Z} 
|a_{d}|^2 |d|^{4+\delta(2+2\eps-2\beta)}
\le 
\sum_{d\in\Z} 
|a_{d}|^2 |d|^{\delta(2+2\eps)}
\le
\|a\|_{F\ell^2_{ \sigma(1+\eps) }}^2.
$$
This proves the assertion.
\end{proof}

{\bf Remark.} 
The condition $\beta>1$ is (in a certain sense) necessary to ensure that $K_\mu(a)$ is trace class.
More precisely, assume that the measure $\mu$ satisfies the condition 
\be\label{cond12}
(\ln m_\xi)''=\frac{\alpha}{\xi^{\beta}}+O(\xi^{-\rho}),\qquad \alpha>0,\ \  1/2<\beta\le 1, \ \ \rho>\beta.
\ee
Choose $\delta>2/\beta>1$. Using Lemma \ref{l3.1}(b) it follows easily that 
$$
\rho_{j,k}=1-\frac{\alpha(j-k)^2}{2(1+j+k)^\beta}(1+o(1))
$$
for indices $(j,k)\in \cI_\delta$. Moreover for each fixed $m$, the entries $(k,m+k)$ belongs to $\cI_\delta$ for all sufficiently large $k\ge k_0(m)$. Thus the $m$-th diagonal 
has entries 
$$
a_m (\rho_{k,k+m}-1)=a_m\frac{\alpha m^2}{2(1+m+2k)^\beta}(1+o(1)),\qquad k\ge k_0(m) .
$$
This growth (in $k$) is too large to allow  $K_\mu(a)$ to be trace class unless $ma_m=0$.
That is, under (\ref{cond12}), the operator $K_\mu(a)$ can only be trace class in the trivial case 
of constant symbol.


\section{Determinant asymptotics}
\label{s4}

Recall that given a function $a\in L^\iy(\T)$ with Fourier coefficients $a_n$, the Toeplitz and the Hankel operator
are defined by their infinite matrix representations 
\bq\label{def.T}
T(a) = (a_{j-k}),\qquad H(a)=(a_{j+k+1}),\qquad 0\le j,k <\iy.
\eq
It is well known that the relations
\bq\label{Tab}
T(ab) &=& T(a)T(b)+H(a)H(\tilde{b}),\\
H(ab) &=& T(a) H(b) + H(a) T(\tilde{b}),\label{Hab}
\eq
hold, where $\tilde{b}(t)=b(t^{-1})$, $t\in \T$. For later  introduce the flip and the projections,
\bq
&&W_n :\{x_0,x_1,\dots\}\mapsto \{x_{n-1},\dots,x_{0},0,0,\dots\},\nn\\
&&P_n\,\,:\{x_0,x_1,\dots\}\mapsto \{x_0,\dots,x_{n-1},0,0,\dots\},\nn
\eq
$Q_n=I-P_n$, and the shift operators
$V_{n}=T(t^{n})$, $ n\in\Z$.

Consistent with previous notation, we denote by $T_{n}(a)$ and $M_{\mu,n}(a)$ the $n \times n$ upper-left submatrices of the matrix representation of $T(a)$ and $M_\mu(a)$, {\em i.e.}, 
$$
T_n(a)=P_n T(a) P_n,\qquad M_{\mu,n}(a)=P_n M_\mu(a) P_n.
$$
Here we identify the upper-left $n\times n$ block in the matrix representation of the operators on the right hand sides with the $\C^{n\times n}$ matrices on the left hand sides.

\medskip
In this section we are going to establish the main auxiliary result (Theorem \ref{mainthm}), which reduces the asymptotics of the determinant
$\det M_{\mu,n}(a)$ to the asymptotics of a trace (or already gives the determinant asymptotics up to the computation of a constant).
This and the main results hold either for the Krein algebra $\K=L^\iy(\T)\cap F\ell^2_{1/2}$ (see \cite[Ch.~10]{BS}), or for several subalgebras of
$C(\T)$, which satisfy ``suitable conditions''. 
Therefore, it seems convenient to formulate Theorem \ref{mainthm} below in a quite general context and to make use of the following  definition.

\begin{definition}\label{d4.1}\em
Given a unital Banach algebra $B$ which is continuously embedded in $L^\iy(\T)$, denote by $\Phi(B)$ the set of all $a\in B$ such that the Toeplitz operator $T(a)$ is invertible on $\ell^2$.  We say such a Banach algebra $B$ {\em suitable} if:
\begin{enumerate}
\item[(a)]
$B$ is continuously embedded in $\K=L^\iy(\T)\cap F\ell^2_{1/2}$.
\item[(b)]
If $a\in \Phi(B)$, then $a\iv \in \Phi(B)$.\qed
\end{enumerate}
\end{definition}

The next proposition demonstrates the suitability of several Banach algebras which appear in the
main results.

\begin{proposition}\label{p4.1}
With $W=F \ell^1_0$ denoting the Wiener algebra, the following are suitable Banach algebras:
\begin{itemize}
\item[(i)]  $W\cap F\ell^2_{\sigma}=F\ell^2_\sigma$ for $\sigma>1/2$;
\item[(ii)] $F\ell^1_{\sigma}$ for $\sigma\ge1/2$;
\item[(iii)] $W\cap F\ell^2_{1/2}$ and $\mathrm{K}=L^\iy(\T)\cap F\ell^2_{1/2}$;
\item[(iv)] $W\cap F\ell^2(\nu)$ provided that $\nu_{-n}=\nu_n\ge n^{1/2}$, $\{\nu_n\}_{n=1}^{\iy}$ is increasing, and
$\sup\limits_{n\ge1}\frac{\nu_{2n}}{\nu_{n}}<\iy$.
\end{itemize}
\end{proposition}

\begin{proof}
First of all, the above are indeed Banach algebras.
This is elementary for $F\ell^1_\sigma$.  A proof for $W\cap F\ell^{2}_{\sigma}$, $\sigma\ge0$ can be found in \cite[Thm.~6.54]{BS}, while the more general space $W\cap F\ell^2(\nu_\sigma)$ is treated in \cite{Ka}. For $\K$ see, e.g., \cite[Thm.~10.9]{BS}. 
As for (i), note that $F\ell^2_\sigma$ is continuously embedded in $W$ whenever $\sigma>1/2$.
Further, property (a) of suitability is immediate for these spaces.

Recall that a unital Banach algebra $B$ is called inverse closed in Banach algebra $B_0\supset B$ 
if $a\in B$ and $a\iv\in B_0$ implies that $a\iv\in B$.
For all the Banach algebras $B$ above, except for $\K$, using simple Gelfand theory and the density of the Laurent polynomials it is easily seen that the maximal ideal space can be naturally identified with $\T$. (In the case of (iv), this is also proved in \cite{Ka}.) By a standard argument, this implies that these Banach algebras are inverse closed in $C(\T)$, thus also in $L^\iy(\T)$. For a proof of the inverse closedness of $\K$ in $L^\iy(\T)$ see again \cite[Thm.~10.9]{BS}.

As for property  (b),  take  $a\in\Phi(B)$, {\em i.e.}, $a\in B$ such that $T(a)$ is invertible on $\ell^2$. 
From the theory of Toeplitz operators it is well known that then $a$ is invertible in $L^\iy(\T)$.
By the inverse closedness we thus have $a\iv\in B$. Now we observe that $b\in \K$ implies that
both $H(b)$ and $H(\tilde{b})$ are Hilbert-Schmidt. Using the formulas
\be\label{TTiv}
I=T(a)T(a\iv)+H(a)H(\tilde{a}\iv),\qquad I=T(a\iv)T(a)+ H(a\iv)H(\tilde{a}),
\ee
and the implied compactness of the Hankel operators, 
it follows that $T(a\iv)$ is a Fredholm regularizer for $T(a)$. (For information about Fredholm operators, see, e.g., \cite{GoKr}.)
Hence $T(a\iv)$ is also Fredholm with index zero and thus invertible (by Coburn's lemma \cite[Sec.~2.6]{BS}). But this means that $a\iv\in \Phi(B)$.
\end{proof}

The next proposition shows (besides a technical result (ii)) that the constant $G[a]$ is well-defined for all $a\in\Phi(B)$. 
This constant appears in our limit theorem as it did appear in the classical Szeg\"o-Widom limit theorem.
 We follow closely the arguments of \cite[Ch.~10]{BS}. 

\begin{proposition}\label{p4.2}
Let $B$ be a suitable Banach algebra, and $a\in \Phi(B)$. 
\begin{enumerate}
\item[(i)]  With $[\,\ast\,]_{00}$  the $(0,0)$-entry of the matrix representation on $\ell^2$, the constant
\be\label{Ga1}
G[a]:= [T^{-1}(a\iv)]_{00}
\ee
is nonzero. 
\item[(ii)]
With  $A_n=P_n T\iv(a\iv) P_n$, we have $\det A_n=G[a]^n$, and
$$
A_n\iv \to T(a\iv),\qquad (A_n^*)\iv\to T(a\iv)^*
$$
strongly on $\ell^2$ as $n\to\iy$. ($A^*$ is the adjoint of $A$). Moreover, the mappings 
$$
\Lambda_n:a\in \Phi(B)\mapsto A_n\iv \in \cL(\ell^2)
$$ 
are equi-continuous.
\item[(iii)]
If $b\in B$, then $e^b\in \Phi(B)$ and $G[e^b]=e^{b_0}$, where $b_0$ is the $0$-th Fourier coefficient.
\end{enumerate}
\end{proposition}
\begin{proof}
(i)-(ii):\
If $a\in \Phi(B)$, then $a\iv \in \Phi(B)$ and hence $T(a\iv)$ is invertible. 
Hence the definitions of $G[a]$ and $A_n$ make sense. Notice that for $n=1$, we have
$\det A_1=A_1=[T\iv(a\iv)]_{00}=G[a]$. Hence (i) will follow from the invertibility of $A_n$ in the case $n=1$.

To show the invertibility of $A_n$ we  use a simple, but useful formula due to Kozak. If $P$ is a projection, $Q=I-P$ is the complementary projection, and $A$ is an invertible operator, then $PAP|_{\mathrm{Im}(P)}$ is invertible if and only if
so is $Q A\iv Q|_{\mathrm{Im}(Q)}$. In fact, the formula
$$
(PAP)|_{\mathrm{Im}(P)}\iv = PA\iv P|_{\mathrm{Im}(P)} -PA\iv Q(QA\iv Q)|_{\mathrm{Im}(Q)}\iv QA\iv P|_{\mathrm{Im}(P)}
$$
holds, which can be easily verified (see also \cite[Prop.~7.15]{BS}). 

Applying Kozak's formula to $A_n=P_nT\iv(a\iv) P_n$ we see that $A_n$ is invertible if and only if $Q_n T(a\iv) Q_n$ is invertible,
and in this case we have
\be\label{An.iv-0}
A_n\iv = P_n T(a\iv) P_n - P_n T(a\iv) Q_n (Q_n T(a\iv )Q_n)\iv Q_n T(a\iv) P_n.
\ee
Notice that $Q_n T(a\iv) Q_n$ is nothing but the ``shifted'' Toeplitz operator. 
Using $V_nV_{-n}=Q_n$, $V_{-n}V_n=I$, we obtain
$(Q_n T(a\iv) Q_n)\iv = V_n T\iv(a\iv) V_{-n}$ and hence
\be
A_n\iv = P_n T(a\iv) P_n - P_n T(a\iv) V_{n} T\iv(a\iv) V_{-n} T(a\iv) P_n.
\label{An.iv}
\ee
We have thus shown that $A_n$ is invertible and in particular (i).  Moreover, from this representation it follows immediately that the mappings
$\Lambda_n$ are equi-continuous. If suffices to remark that the operators $P_n$ and $V_{\pm n}$ have norm one, and that the various mappings $b\in\Phi(B)\mapsto b\iv\in\Phi(B)$, $b\in B\mapsto T(b)\in \cL(\ell^2)$, $B\in G\cL(\ell^2)\mapsto B\iv\in\cL(\ell^2)$ are continuous.
Using that $P_n=P_n^*\to I$ strongly, and $V_{n}^*=V_{-n}\to 0$ strongly on $\ell^2$, it follows that $A_n\iv$ and their adjoints
converge strongly.

In order to prove $\det P_n T\iv(a\iv)P_n=G[a]^n$ is suffices to prove that 
\be\label{f.Cram}
\frac{\det A_n}{\det A_{n-1}}=G[a]
\ee
for $n\ge 1$. For $n=1$ with $\det A_0:=1$, this is just the definition of $G[a]$. By noting that $A_{n-1}=P_{n-1}A_nP_{n-1}$ it follows from Cramer's rule that 
$$
\frac{\det A_{n-1}}{\det A_n}=[A_n\iv]_{n-1,n-1}
$$
for $n\ge2$ while the statement is obvious for $n=1$.
Reformulating the above expression (\ref{An.iv}) for $A_n\iv$ one step further, we have
\be\label{An.iv2}
A_n\iv = W_nT(\tilde{a}\iv)W_n-W_nH(\ta\iv)T\iv(a\iv)H(a\iv)W_n=W_nT\iv(\tilde{a})W_n.
\ee
Here we use the general formulas 
$$
P_nT(b)P_n=W_n T(\tilde{b})W_n,\quad 
P_nT(b)V_n=W_nH(\tilde{b}),\quad
V_{-n}T(b)P_n=H(b)W_n.
$$
as well as an identity relating the inverses of $T(a\iv)$ and $T(\tilde{a})$ to each other (which either can be derived from Kozak's formula or by using (\ref{Tab}), (\ref{Hab})).
Due to the definition of the $W_n$, we see that the lower-right entry of $A_n\iv$ does not depend on $n$ for $n\ge1$, {\em i.e.}, 
$$
[A_n\iv]_{n-1,n-1}=[T\iv(\tilde{a})]_{00}=1/G[a],
$$
the last equality following from (\ref{f.Cram}) for $n=1$.
This completes the proof of (\ref{f.Cram}) for all $n$.

(iii):\  Using (\ref{TTiv}) it can be seen that $T(e^{-\lambda b})$ is a Fredholm regularizer of
$T(e^{\lambda b})$, $\lambda\in[0,1]$. Due to the stability of the Fredholm index under perturbation,
all these operators have Fredholm index zero; hence they are invertible (Coburn's lemma \cite[Sec.~2.6]{BS}). This proves $e^b\in\Phi(B)$.
A proof of $G[e^b]=e^{b_0}$ can now be given via an approximation argument and by using Wiener-Hopf factorization (see \cite[Prop.~10.4]{BS}).
\end{proof}

Before stating the main result of this section, we introduce two conditions on a Banach algebra $B\subseteq L^\iy(\T)$.

\begin{enumerate}
\item[{\bf (TC)}]
For all $a\in B$ the operator $K_\mu(a)$ is trace class and 
$\|K_\mu(a)\|_{\cC_1(\ell^2)} \le C \|a\|_{B}.$
\item[{\bf (HS)}]
For all $a\in B$ the operator $K_\mu(a)$ is Hilbert-Schmidt and
$
\|K_\mu(a)\|_{\cC_2(\ell^2)} \le C \|a\|_{B}.
$
\end{enumerate}
Propositions \ref{p3.2} and \ref{p3.3} identify Banach algebras $B$ which satisfy the above, the criteria
involving the underlying measure $\mu$ (the constant $C$ depends on $\mu$).

\begin{theorem}
\label{mainthm}
Let $B\subset L^\iy(\T)$ be a suitable Banach algebra.
\begin{enumerate}
\item[(a)] Suppose $B$ satisfies {\bf (TC)}. Then for $a\in \Phi(B)$ we have
\bq\label{eq19}
\lim_{n\to\iy}\frac{\det M_{\mu,n}(a)}{G[a]^n}=E[a]
\eq
where
$$
E[a]=\det\Big( T(a\iv) M_\mu(a)\Big).
$$
The constant $E[a]$ is a well-defined operator determinant, and the convergence (\ref{eq19})
is uniform in $a\in\Phi(B)$ on compact subsets of $\Phi(B)$.

\item[(b)]
Suppose $B$ satisfies {\bf (HS)}. Then for $a\in \Phi(B)$ we have
\bq\label{eq20}
\lim_{n\to\iy} \frac{\det M_{\mu,n}(a)}{G[a]^n\cdot  \exp(\trace P_n T(a\iv)K_\mu(a)P_n)}=H[a]
\eq
with 
$$
H[a]=\det\Big(T(a\iv)M_\mu(a)e^{-T(a\iv)K_\mu(a)}\Big).
$$
Again, the constant $H[a]$ is a well-defined operator determinant, and the convergence
(\ref{eq20}) is uniform in $a\in \Phi(B)$ on compact subsets of $\Phi(B)$.
\end{enumerate}
\end{theorem}
\begin{proof}
The first steps in the proof of (a) and (b) are the same.  As in the previous proposition define
$A_n=P_n T\iv(a\iv) P_n$. Recall (\ref{TTiv}) to conclude that 
$$
T(a)= T\iv(a\iv)+L(a),\qquad L(a):=-T\iv(a\iv)H(a\iv)H(\ta)
$$
with $L(a)$ being trace class. The latter follows from the fact that $H(b)$ and $H(\tb)$ are Hilbert-Schmidt for $b\in B\subseteq \K$,
while appropriate norm estimates also hold. Moreover, property (b) of the suitability of $B$ implies that the mapping
$$
a\in\Phi(B)\mapsto L(a)\in  \cC_1(\ell^2)
$$ 
is continuous. Now we can write
\bq
M_{\mu,n}(a) &=& P_n \left(T\iv(a\iv)+L(a)+K_\mu(a)\right) P_n\nn\\
&=& A_n +P_n(L(a)+K_\mu(a))P_n.\nn
\eq
Using Proposition \ref{p4.2}(ii) we obtain
\be\label{frac.det}
\frac{\det M_{\mu,n}(a)}{G[a]^n}=\det\left(P_n+A_n\iv P_n(L(a)+K_\mu(a))P_n  \right).
\ee

(a):\  Assume condition {\bf (TC)}. Then $K_\mu(a)$ is trace class, and the mapping $a\in\Phi(a)\mapsto
K_\mu(a)\in\cC_1(\ell^2)$ is continuous.  Consequently, again by Proposition \ref{p4.2}(ii),
$$
\det\left(P_n+A_n\iv P_n(L(a)+K_\mu(a))P_n  \right)
$$
converges to the well defined operator determinant
$$
\det\left(I+T(a\iv)(L(a)+K_\mu(a)) \right),
$$
which equals 
$$
\det\left(T(a\iv)(T(a)+K_\mu(a))\right)=\det\left( T(a\iv)M_\mu(a)\right).
$$
As to the uniform convergence on compact subset of $\Phi(B)$, it is enough to show that the family of maps
$$
 a\in \Phi(B) \mapsto \det\left(P_n + A_n\iv P_n(L(a)+K_\mu(a))P_n  \right)\in \C
$$
are equi-continuous. To see this we use the equi-continuity of $a\in \Phi(B) \mapsto A_n\iv\in \cL(\ell^2)$
and the continuity of $a\in \Phi(B)\mapsto L(a)+K_\mu(a)\in \cC_1(\ell^2)$ along with fact that $\sup \|A_n\iv\|<\iy$ for each $a\in \Phi(B)$. This implies that the maps
$$
a\in \Phi(B)\mapsto A_n\iv P_n(L+K_\mu(a)) P_n
$$
are equi-continuous and bounded. Finally, in order to pass to the determinant we use the general estimate
$$
|\det(I+A)-\det(I+C)|\le \|A-C\|_{1}\exp\left(\max\{\|A\|_1,\|C\|_1\}\right),
$$
which holds for trace class operators $A,C$.

(b):\
Now assume condition {\bf (HS)}. 
In view of (\ref{frac.det}) introduce
$$
C_n = A_n\iv P_n(L(a)+K_\mu(a)) P_n.
$$
Then 
$$
C_n = A_n\iv P_n L(a) P_n +P_n T(a\iv) K_\mu(a) P_n+ D_n
$$
with
$$
D_n  =
(A_n\iv P_n - P_n T(a\iv)) K_\mu(a) P_n.
$$
{}From (\ref{An.iv-0}) and $P_n=I-Q_n$ we obtain
\bq
A_n\iv P_n  -P_n T(a\iv) &=& -P_nT(a\iv)Q_n  - P_n T(a\iv) Q_n (Q_nT(a\iv)Q_n)\iv Q_n  T(a\iv) (I-Q_n)
\nn\\
&=&  - P_n T(a\iv) Q_n (Q_nT(a\iv)Q_n)\iv Q_n  T(a\iv)\nn
\eq
Using the same arguments as in the derivation of (\ref{An.iv}) and (\ref{An.iv2}), this equals
$$
- W_n H(\tilde{a}\iv) T\iv(a\iv) V_{-n} T(a\iv),
$$
whence
$$
D_n= -W_n H(\tilde{a}\iv) T\iv(a\iv) V_{-n}T(a\iv) K_\mu(a)P_n.\nn
$$
Since $H(\tilde{a}\iv)$ and $K_\mu(a)$ are each Hilbert-Schmidt, and $V_{-n}\to 0$ strongly, it follows that 
$D_n\to0$ in the trace norm. Moreover, from the explicit representation it is seen 
that the family of mappings $a\in \Phi(B) \mapsto D_n\in \cC_1(\ell^2)$ is equi-continuous.

Further, by Proposition \ref{p4.2}(ii), $A_n\iv P_n L(a) P_n\to  T(a\iv) L(a)$ converges
in the trace norm, and the family of maps $a\in\Phi(B)\mapsto A_n\iv P_n L(a)P_n\in\cC_1(\ell^2)$ 
is equi-continuous.

In contrast, $P_nT(a\iv) K_\mu(a)P_n$ converges only in the Hilbert-Schmidt norm
to $T(a\iv) K_\mu(a)$, while  the mappings $a\in\Phi(B)\mapsto
P_nT(a\iv)K_\mu(a)P_n\in\cC_2(\ell^2)$ are equi-continuous. 

We can now conclude that on each compact subset of $\Phi(B)$, the afore-mentioned maps
are actually uniformly equi-continuous and uniformly bounded. Hence we have uniform convergence
of the corresponding sequences of operators in the trace class or Hilbert-Schmidt norm.

With $C=T(a\iv)L(a) + T(a\iv) K_\mu(a)=T(a\iv) M_\mu(a)$, noting that $L(a)=T(a)-T(a\iv)\iv$,
 it follows that, as $n \rightarrow \infty$,
$$
(I+C_n)e^{-P_nT(a\iv) K_\mu(a) P_n}-I  \rightarrow (I+C)e^{-T(a\iv) K_\mu(a)}-I,
$$
uniformly on compact subset of $\Phi(B)$ in  trace norm. 
Consequently,
$$
  \lim_{n \rightarrow \infty} \det\left((I+C_n)e^{-P_n T(a\iv) K_\mu(a) P_n}\right)
=
\det \left((I+C)e^{-T(a\iv) K_\mu(a)}\right),
$$
also uniformly.
\end{proof}

Let us summarize what we have achieved thus far:

\medskip

Assuming the moment condition (C1), {\em i.e.}, ``$\beta>1$'',  we have both the trace class condition {\bf (TC)}
and the  Hilbert-Schmidt condition {\bf (HS)} available (see Proposition \ref{p3.2} and \ref{p3.3}).
The easiest way is to assume {\bf (TC)} and use
Theorem \ref{mainthm}(a) to conclude a limit theorem. However, the trace class conditions are much stronger
than the Hilbert-Schmidt conditions, and it is worthwhile to see what can be done assuming only the latter. Then we can apply
Theorem \ref{mainthm}(b), and are left with the computation of traces (which will be done in
Proposition \ref{p5.1} below). While we get a better result assuming only {\bf (HS)},  
the constant expression will be more complicated.

\medskip

Assuming the moment condition (C2), {\em i.e.}, ``$1/2<\beta\le 1$'', $K_\mu(a)$ will in general not be trace class
(see the remark at the end of Section \ref{s3}). Therefore we are left with Theorem \ref{mainthm}(b)
and the computation of the traces, which in this case is more diffucult and will occupy most of the next section.


\section{Asymptotics of the trace}
\label{s5}

As just pointed out, in order to make use of part (b) of Theorem \ref{mainthm}, we need to evaluate the
trace term. We distinguish between the two cases indicated above.

The case of $\beta>1$ is completely settled by the following proposition, which shows that the trace converges to a constant.

\begin{proposition}\label{p5.1}
Assume the moment condition {\em (C1)}, and put $\sigma=1/2\vee 1/\beta$. Then, for $a,b\in F\ell^2_\sigma$, we have
\bq\label{f.24}
  \trace \left(P_n T(b)K_{\mu}(a) P_n\right) &=& \tau_\mu(a,b) + o(1), \qquad n\to\iy,
\eq
where
\bq\label{f.25}
\tau_\mu(a,b) &:=&\sum_{j,k=0}^\iy b_{k-j}a_{j-k}(\rho_{j,k}-1).
\eq
The series (\ref{f.25}) converges absolutely. Moreover, the convergence (\ref{f.24}) is uniform in $(a,b)$ on compact subsets of $F\ell^2_\sigma\times F\ell^2_{\sigma}$.
\end{proposition}
\begin{proof}
By Proposition \ref{p3.2} the operator $K_\mu(a)$ is a Hilbert-Schmidt and hence bounded and linear. 
Consequently the trace equals
$$
\trace\left( P_n T(b)K_{\mu}(a) P_n\right) =\sum_{j=0}^\iy\sum_{k=0}^{n-1} b_{k-j}a_{j-k}(\rho_{j,k}-1).
$$
We claim that the estimate
\be\label{F0}
\sum_{j,k=0}^\iy \left|b_{k-j} a_{j-k}(\rho_{j,k}-1)\right|\le C\|a\|_{F\ell^2_\sigma}\|b\|_{F\ell^2_\sigma}
\ee
holds. Indeed, put $\delta=2\sigma=1\vee 2/\beta$, recall $0<\rho_{j,k}\le1$, and split the sum into
$$
\sum_{(j,k)\notin \cI_\delta} |b_{k-j}a_{j-k}|+\sum_{(j,k)\in\cI_\delta}|b_{j-k}a_{j-k}(\rho_{j,k}-1)|,
$$
where $\cI_\delta$ is defined in (\ref{I.delta}).
Using Lemma \ref{l3.1}(a) and substituting  $m=j-k$ and $\ell=j+k$ we can overestimate this by
$$
\sum_{(m,\ell) \in \Z\times\Z_+ \atop 2|m|^\delta\ge \ell}|b_{-m}a_m| +
\sum_{(m,\ell) \in \Z\times\Z_+ \atop 2|m|^\delta <  \ell}|b_{-m}a_m|\frac{m^2}{\ell^{\beta}}
\le C
\sum_{m=-\iy}^\iy|b_{-m}a_m| |m|^\delta+ C \sum_{m=-\iy}^\iy|b_{-m}a_m| |m|^{2+\delta(1-\beta)}.
$$
From Cauchy's inequality and since $\delta\beta\ge 2$, we  obtain (\ref{F0}).

The convergence (\ref{f.24}) of the trace now follows from (\ref{F0}) by dominated convergence. The absolute convergence of
(\ref{f.25}) is also a consequence of (\ref{F0}). Finally, again by (\ref{F0}), the mappings
$$
\Lambda_n:(a,b)\in F\ell^2_\sigma\times F\ell^2_\sigma\mapsto \trace(P_n T(b)K_\mu(a)P_n),\quad n\ge 1
$$
are equi-continuous. Convergence and equi-continuity imply the uniform convergence on compact subsets. 
\end{proof}

We remark that the function $\tau_\mu(a,b)$ is bilinear and continuous in $a,b\in F\ell^2_\sigma$.
Formally $\tau_\mu(a,b)$ equals the trace of $T(b)K_\mu(a)$, though note
the assumptions made in the proposition are not sufficient to insure
$T(b)K_\mu(a)$ is trace class.
Indeed, there exists $a\in F\ell^2_\sigma$ such that $K_\mu(a)$ is not trace class 
(and one can choose $b=1$). Of course, if $K_\mu(a)$ is trace class, we have equality
(and the proposition is a triviality).

Now we turn to the case $1/2<\beta \le 1$,  for which the trace does not converge
to a constant. It provides the second order asymptotics of the $\det M_{\mu,n}(a)$. In terms
of the random matrix interpretation, the asymptotics of the trace gives the
shape of the variance for the corresponding linear statistics.  We begin with  the following estimate.

\begin{lemma} \label{l5.2}
Assume the moment condition {\em (C2)}, and put $\delta=2\sigma=2/\beta\vee 3/\gamma$.
Then for  $a,b\in F \ell^2_{\sigma}$ it holds
\be\label{tr1}
\trace\left( P_n T(b)K_{\mu}(a) P_n\right) =-\frac{1}{2} \sum_{m=-\iy}^\iy  
m^2 b_{-m}a_{m} p_{n,m}^{(\delta)} + {E_1}(a,b;\delta) + o(1), \qquad n\to\iy.
\ee
Here $E_1$ is  constant
and 
\be\label{p.delta}
p_{n,m}^{(\delta)} =  \sideset{}{'} \sum_{2|m|^\delta<\ell \le 2n} h_\mu(\ell),
\ee
where the prime indicates that the summation is taken over all $\ell\in\Z_+$ with the same parity as $m$. The convergence (\ref{tr1}) is uniform in 
$(a,b)$ on compact subsets of $F\ell^2_{\sigma}\times F\ell^2_{\sigma}$. 
\end{lemma}
\begin{proof}
As in the previous lemma, the operator $K_\mu(a)$ is Hilbert-Schmidt and  the trace evaluates to 
$$
\trace\left( P_n T(b)K_{\mu}(a) P_n\right) =\sum_{j=0}^\iy\sum_{k=0}^{n-1} b_{k-j}a_{j-k}(\rho_{j,k}-1).
$$
We can split the double series into 
\be\label{F1}
\sum_{(j,k)\notin \cI_\delta \atop k<n} b_{k-j}a_{j-k}(\rho_{j,k}-1)\quad \mbox{ and } \quad
\sum_{(j,k)\in \cI_\delta \atop k<n} b_{k-j}a_{j-k}(\rho_{j,k}-1),
\ee
where the first term is dominated by
$$
\sum_{(j,k)\notin \cI_\delta} |b_{k-j}a_{j-k}|\le C\|a\|_{F\ell^2_{\sigma}}\|b\|_{F\ell^2_{\sigma}}.
$$
Consequently, the first term in (\ref{F1}) converges as $n\to\iy$  to the constant
\bq
\sum_{(j,k)\notin\cI_\delta} b_{k-j}a_{j-k}(\rho_{j,k}-1),
\label{E1ab}
\eq
and using equi-continuity we see that the  convergence is uniform on compact subsets.

For the second term in (\ref{F1}) we bring in the estimate of Lemma \ref{l3.1}(b),
$$
\rho_{j,k} = 1 - \frac{m^2}{2}h_\mu(\ell) + O \left( \frac{m^4}{\ell^{2 \beta}} \vee
\frac{|m|^3}{\ell^{\gamma}}\vee \frac{m^2}{\ell^{\rho}}\right),
\qquad (j,k)\in\cI_\delta,
$$
together with the substitution $\ell=j+k$, $m=j-k$. As to the applicability of this lemma, note that
$\delta\rho>\delta\ge2/\beta\ge2$. Hence the second term in (\ref{F1}) equals
\be\label{E12}
-\sum_{(j,k)\in \cI_\delta \atop k<n} 
b_{-m}a_{m} \frac{m^2}{2} h_\mu(\ell)
+ \sum_{(j,k)\in \cI_\delta \atop k<n} 
b_{-m}a_{m} 
 O \left( \frac{m^4}{\ell^{2 \beta}} \vee
\frac{|m|^3}{\ell^{\gamma}}\vee \frac{m^2}{\ell^{\rho}}\right).
\ee
The error term here can be overestimated by a constant multiple of
$$
 \sum_{m\in\Z} |b_{-m}a_m| \cdot  ( |m|^{4+\delta(1-2\beta)} \vee  |m|^{3+\delta(1-\gamma)} \vee  |m|^{2+\delta(1-\rho)} )
\le  \|a\|_{F\ell^2_\sigma}\|b\|_{F\ell^2_\sigma}.
$$
Here, we first converted the sum over $(j,k)$ to that over $(m,\ell) \in \Z \times \Z_{+}$ restricted to $2 |m|^{\delta} < \ell$ and then summed over 
the $\ell$ variable.  After this one notes that our conditions imply that the exponents
$4+\delta(1-2\beta)$, $3+\delta(1-\gamma)$, and $2+\delta(1-\rho)$ are all less than $\delta = 2 \sigma$.
In other words, the error  in (\ref{E12}) is dominated by
a corresponding absolutely convergent series. As such it  converges to the  constant 
\bq
\sum_{(j,k)\in\cI_\delta}b_{-m}a_{m}\left(\rho_{j,k}-1+\frac{m^2}{2}h_\mu(\ell)\right)
\label{E2ab}
\eq
as $n\to\infty$. 
In fact, the convergence is uniform on compact subsets of $F\ell^2_\sigma\times F\ell^2_\sigma$, 
which can be most  easily seen by equi-continuity.  In view of what follows, the constant $E_1(a,b, \delta)$ is now identified as the sum of (\ref{E1ab}) and (\ref{E2ab}).

Turning to the first term in (\ref{E12}), the summation expressed in terms of $(m,\ell)\in
\Z\times \Z_+$ is over all indices such that 
$ \ell<2n+m,$ $2|m|^\delta< \ell, $
and such that the parity of $\ell$ and $m$ is the same.
That is,  what we have for the leading order is 
\be\label{t1}
\sum_{m=-\iy}^\iy b_{-m}a_m\frac{m^2}{2}
\left( \sideset{}{'}\sum_{2|m|^\delta<\ell < 2n+m} 
h_\mu(\ell)\right)
\ee
while 
\be\label{t2}
\sum_{m=-\iy}^\iy b_{-m}a_m \frac{m^2}{2}
\left( \sideset{}{'} \sum_{2|m|^\delta<\ell \le  2n} h_\mu(\ell)\right)
\ee
is what is claimed in (\ref{tr1}).

We next show that 
\be
s_{n,m}:=
 \sideset{}{'}\sum_{2|m|^\delta<\ell < 2n+m} m^2 h_\mu(\ell)-
 \sideset{}{'} \sum_{2|m|^\delta<\ell \le 2n} m^2h_\mu(\ell)=O\left(\frac{|m|^\delta}{n^{\eps}}\vee \frac{|m|^\delta}{n^\beta}\right),
\label{snm}
\ee
as $n\to\iy$, uniformly in $m$, where $\eps =\beta+1-3/\delta>0$.   This will imply that the difference between
(\ref{t1}) and (\ref{t2}) converges (uniformly) to zero as $n\to\iy$.

To see (\ref{snm}) we distinguish four cases:
\begin{enumerate}
\item
$m>0$ and $2|m|^\delta<2n$. Then $s_{n,m}=O(m^3/n^\beta)$. Since $m<n^{1/\delta}$ we have
$$
\frac{m^3}{n^\beta}\le \frac{m^{\delta} n^{(3-\delta)/\delta}}{n^\beta}=\frac{m^\delta}{n^\eps}
$$
in case $\delta<3$, while the bound is $m^\delta/n^\beta$ in the case $\delta\ge3$.
\item
$m>0$ and $2n\le 2 |m|^\delta$. Then $s_{n,m}=O(m^3/m^{\beta\delta})$, and since  $m\ge n^{1/\delta}$,
we have 
$$\frac{m^3}{m^{\beta\delta}}=\frac{m^\delta}{m^{\beta\delta+\delta-3}}\le 
\frac{m^\delta}{n^{\beta+1-3/\delta}}=\frac{m^\delta}{n^\eps}.
$$
\item
$m<0$ and $2|m|^\delta<2n+m$. Then $s_{n,m}=O(|m|^3/(2n-|m|)^\beta)$, $|m|<(n-|m|/2)^{1/\delta}\le n^{1/\delta}$,
and we have
$$
\frac{|m|^3}{(n-|m|/2)^\beta}\le \frac{|m|^\delta(n-|m|/2)^{(3-\delta)/\delta}}{(n-|m|/2)^{\beta}}
=\frac{|m|^\delta}{(n-|m|/2)^{\eps}}\le
\frac{|m|^\delta}{(n-n^{1/\delta}/2)^{\eps}}
$$
in case $\delta<3$, or  $|m|^\delta/(n-n^{1/\delta}/2)^\beta$ in the case $\delta\ge3$.
\item
$m<0$ and $2n+m\le 2|m|^\delta$. Then $s_{n,m}=O(|m|^3/|m|^{\beta\delta})$,
$n\le |m|^\delta+|m|/2\le 2|m|^\delta$, and
$$
\frac{|m|^3}{|m|^{\beta\delta}}=\frac{|m|^\delta}{|m|^{\beta\delta+\delta-3}}\le C \frac{|m|^\delta}{n^{\beta+1-3/\delta}}=C\frac{|m|^\delta}{n^\eps}.
$$
\end{enumerate}
From here it follows that difference of (\ref{t1}) and (\ref{t2}) is bounded by a constant multiple of
$
{n^{-\eps \wedge \beta }} \|a\|_{F\ell^2_\sigma}\|b\|_{F\ell^2_\sigma},
$
and  the indicated convergence is uniform in $(a,b)$ even on bounded subsets of $F\ell^2_\sigma\times F\ell^2_\sigma$. 
The proof is finished.
\end{proof}


Next we estimate the leading term from the previous lemma.

\begin{lemma}\label{l5.3}
Assume the moment assumption {\em (C2)}, and define $p_{n,m}^{(\delta)}$ for $\delta>1$ by (\ref{p.delta}). 
\begin{itemize}
\item[(i)]
If $c\in W=F\ell^1$, then 
\be\label{tr.as1}
\sum_{m=-\iy}^\iy c_m p_{n,m}^{(\delta)}=
\iota_\mu(2n) \sum_{m=-\iy}^\iy c_m +o(\iota_\mu(2n)),\qquad n\to\iy.
\ee
\item[(ii)]
If $c\in F\ell^1(\hat{\nu})$ with $\hat{\nu}_m=1+\iota_\mu(2|m|^\delta)$, then, with some constant $E_2$,
\be\label{tr.as2}
\sum_{m=-\iy}^\iy c_m p_{n,m}^{(\delta)}= \iota_\mu(2n) \sum_{m=-\iy}^\iy c_m  +  E_2(c; \delta) +o(1),
\qquad n\to \iy.
\ee
\end{itemize}
The convergence holds uniformly in $c$ on compact subsets of $W$ and $ F\ell^2(\hat{\nu})$, respectively.
\end{lemma}
\begin{proof}
First set 
$$
s_\mu^\pm (x)=  \sum_{{1\le  \ell\le x}\atop (-1)^\ell =\pm 1} h_\mu(\ell).
$$
Standard estimates  using the assumptions on $h_{\mu}$  and the fact that the functions $s_\mu^\pm(x)$ are increasing gives
$s_\mu^\pm(x)= \iota_\mu(x)+C_\pm+o(1)$ as $x\to\iy$ for constants $C_\pm$.
Granted this, for either  
point  (i) or (ii), we split the sum over even and odd indices.
In particular,
\bq
\sum_{m\ \mathrm{even}} c_m p_{n,m}^{(\delta)}
&=& 
\sum_{m\ \mathrm{even}}  c_m \max\left\{ 0,s^+_\mu(2n)-s^+_\mu(2|m|^\delta)\right\} \nn \\
&=&
s_\mu^+(2n) \sum_{m\ \mathrm{even}}  c_m - \sum_{m\ \mathrm{even}} c_m \min \left\{ s^+_\mu(2n),s^+_\mu(2|m|^\delta)\right\}. \nn
\eq
The first term on the right hand side gives one half of the leading asymptotics.
Next we show that for part (i), the second term is $o(s_\mu^+(2n))$, while for part (ii)
the second term is a constant plus $o(1)$. 

Indeed, for part (i), we write the second term as
$$
s_\mu^+(2n)\sum_{m\ \mathrm{even}}  c_m \min\left\{1,\frac{s_\mu^+(2|m|^\delta)}{s_\mu^+(2n)}\right\}.
$$
This renormalized series is dominated by the series  $ \sum |c_m|$. Moreover, for each fixed $m$, the minimum converges to zero as $n\to\iy$. Dominated convergence then implies that the series is $o(1)$ as $n\to\iy$.
Similar considerations can be carried out for the odd term,
concluding the proof of part (i).

As for part (ii), take again the even terms:
$$
\sum_{m\ \mathrm{even}} c_m \min \left\{ s^+_\mu(2n),s^+_\mu(2|m|^\delta)\right\}.
$$
This sum is now dominated by (a constant times)
\be\label{extra.cond}
\sum_{m=-\iy}^\iy |c_m|\left(1+ \iota_\mu(2|m|^\delta)\right)<\iy,
\ee
while for each fixed
$m$, the minimum converges to $s_\mu^+(2|m|^\delta)$ as $n\to\iy$. So dominated convergence yields that the above equals
$$
\sum_{m\ \mathrm{even}} c_m s^+_\mu(2|m|^\delta)+o(1).
$$
The terms involving the summation over odd $m$ give a similar contribution, and collecting everything
we arrive at, in case (ii):
\bq
\lefteqn{\sum c_m p_{n,m}^{(\delta)}} \nn \\
& = &
\sum_{m\ \mathrm{even}} c_m  \left( s_\mu^+(2n) -  s_\mu^+(2|m|^\delta) \right) +
\sum_{m\ \mathrm{odd}} c_m   \left(  s_\mu^-(2n) -s_\mu^-(2|m|^\delta)  \right) +o(1). \nn
\eq
From here the constant
$$
E_2(c;\delta) =
C_+\sum_{m\ \mathrm{even}} c_m +
C_-\sum_{m\ \mathrm{odd}} c_m 
 -\sum_{m=-\infty}^{\infty} c_m  \sideset{}{'}\sum_{1\le \ell \le 2 |m|^{\delta}} h_{\mu} (\ell)
$$
is identified.
The uniform convergence on compacts is seen by using the equi-continuity of the corresponding mappings. 
\end{proof}

We now combine the previous two lemmas into the following theorem. Notice that part (i) will be used to prove Theorem \ref{thm1},
while part (ii) is used to show Theorem \ref{thm3}(a).

\begin{theorem}\label{t5.4}
Assume the moment condition {\em (C2)}, and put $\sigma=1/\beta\vee 3/(2\gamma)$.

\begin{itemize}
\item[(i)]
If $a,b\in F\ell^2_\sigma$, then 
\be\label{d.0}
\trace( P_n T(b) K_{\mu}(a) P_n) =     \Omega(a,b) \cdot \iota_\mu(2n)  +o(\iota_\mu(2n)),
\qquad n\to\iy,
\ee
where 
$$
\Omega(a,b)= -\frac{1}{2}\sum_{m=-\iy}^\iy m^2 a_m b_{-m}=- \frac{1}{4 \pi}\int_{0}^{2\pi} a'(e^{it})b'(e^{it}) dt, 
$$
and the convergence (\ref{d.0}) is uniform in $(a, b)$ on compact subsets of $F \ell_{\sigma}^2\times F \ell_{\sigma}^2$.
\item[(ii)] Let $B=F\ell^2_\sigma\cap F\ell^2(\nu)$ with $\nu_m=\sqrt{1+ m^2\iota_\mu(2|m|^{2\sigma})}$.
Then, for 
$a, b\in B $,
\be\label{d.1b}
\trace( P_n T(b) K_{\mu}(a) P_n) =     \Omega(a, b) \cdot \iota_\mu(2n)  + C_\mu(a, b)+o(1),
\qquad n\to\iy,
\ee
with a certain constant $C_{\mu}(a,b)$. The convergence   (\ref{d.0}) is uniform in $(a, b)$ on compact subsets of $B\times B$.
\end{itemize}
\end{theorem}
\begin{proof}
(i):\ We employ Lemma \ref{l5.2} and Lemma \ref{l5.3}(i) with $c_m=m^2 b_{-m} a_m$ and $\delta=2\sigma$. Since $\sigma\ge 1/\beta\ge 1$, we obtain from Cauchy-Schwartz that $c\in F\ell^1_{2\sigma-2}\subseteq W$. Hence
$$
\trace( P_n T(b) K_{\mu}(a) P_n) = -\frac{\iota_\mu(2n)}{2}\sum_{m=-\iy}^\iy c_m+o(\iota_\mu(2n)),\qquad n\to\iy,
$$
with the convergence being uniform in $a,b$ on compact subsets of $F\ell^2_\sigma$. The computation of the constant $\Omega(a,b)$ is straightforward.

(ii):\  Lemma \ref{l5.2} is applied without any change. This produces the constant factor $E_1$ 
which could be neglected in case (i).
Lemma \ref{l5.3}(ii) is now applicable because $a,b\in F\ell^2(\nu)$ along with Cauchy-Schwartz implies that $c\in F\ell^1(\hat{\nu})$.
We thus obtain the asymptotics (\ref{tr.as2}).
Combined with Lemma \ref{l5.2} we arrive at  (\ref{d.1b}) with the overall constant evaluated from
$E_1$ and $E_2$,
\bq
 C_{\mu}(a,b) & = & \sum_{j,k=0}^\iy b_{k-j}a_{j-k}\left(\rho_{j,k}-1+\frac{(j-k)^2}{2}h_\mu(j+k)\right)\nn \\
&  & 
-\frac{C_+}{2}\sum_{m\ \mathrm{even}} m^2 a_m b_{-m}-\frac{C_-}{2}\sum_{m\ \mathrm{odd}} m^2 a_m b_{-m}
\label{f.54}.
\eq
The constant $C_\pm$ were defined at the beginning of the proof of Lemma \ref{l5.3}.
The absolute convergence of the above series is, among other things, guaranteed by estimates on $a_m$ and $b_m$ that
follow from the choice of $B$.
\end{proof}


\section{Limit theorems: the case $\boldsymbol{\beta>1}$ (C1)}
\label{s6}

We  are now going to give the proof of the main results stated in the introduction in the cases where the moment condition (C1) is fulfilled, {\em i.e.}, $\beta>1$. 

As already pointed out at the end of Section \ref{s4}, we can proceed in two ways, by using either Theorem \ref{mainthm} (a) or (b) depending whether we have the trace class {\bf (TC)} or Hilbert-Schmidt {\bf (HS)} condition available. Sufficient criteria for these condition to hold
are identified in Propositions \ref{p3.2} and \ref{p3.3}. We start with the proof of Theorem \ref{thm3}(b).

Let us first proceed the simpler way.
Put $ B=F\ell^1_{\sigma}$, or $B=F\ell^2_{\sigma+\epsilon}$, $\epsilon>0$ with $\sigma = 1\vee 2/\beta$. Then Proposition \ref{p3.3} implies that $B$ satisfies the trace class condition  {\bf (TC)},
and   Proposition \ref{p4.1} shows that the Banach algebra $B$ is suitable.
Now  apply Theorem \ref{mainthm}(a) in order to get (\ref{con.2}) in Theorem \ref{thm3}(b). In particular, we obtain the correct identification of the constant $E[a]$ as
a well-defined operator determinant. As for the constant $G[a]$, see Proposition \ref{p4.2}(i) and (iii),
noting that (because $B\subset C(\T)$) formula (\ref{Ga1}) reduces to (\ref{Ga}).

Proceeding the other way, put $B=L^\iy(\T)\cap F\ell^2_{1/2}$ ($\beta\ge 2$) or
$B=F\ell^2_{1/\beta}$ ($1<\beta<2$). Again suitability of $B$ is guaranteed by Proposition \ref{p4.1},
and Proposition \ref{p3.2} implies  {\bf (HS)}. Now we can use 
Theorem \ref{mainthm}(b), and we are left with the asymptotics of the trace, which is settled by
Proposition \ref{p5.1}. We obtain the same convergence (\ref{con.2}) in Theorem \ref{thm3}(b)
under the stated (more general) conditions,  but the
constant $E[a]$ must be identified as
$$
E[a]=e^{\tau_\mu(a,a\iv)}\det\Big( T(a\iv)M_\mu(a) e^{-T(a\iv)K_\mu(a)}\Big).
$$
Clearly, if $a$ satisfies the stronger conditions, then both expressions for $E[a]$ coincide (see also the remark after Proposition \ref{p5.1}).
This concludes the proof of Theorem \ref{thm3}(b).

\medskip
For our random matrix application (Theorem \ref{thm2}), the behavior of the  (centered) linear statistic $X_{f,n} - n f_0=X_{f-f_0,n}$ is accessed
through considering symbols $a_{\lambda} = e^{i \lambda (f-f_0)}$. Notice that  Proposition \ref{p4.2}(iii)
implies $a_{\lambda}\in \Phi(B)$ and $G[a_{\lambda}]=1$.
Applying what we have just proved (Theorem \ref{thm3}(b))  and (\ref{detform}) we immediately obtain
\be\label{E.conv}
  \lim_{n \to \infty}  \E_{\m,n} \left[ e^{ i \lambda (X_{f,n} - n f_0) } \right]  = 
  E(f, \lambda)
\ee
with
\be\label{Ela}
E(f, \lambda):= e^{\tau_\mu(a_\lambda\iv,a_\lambda)}\det\Big( T(a_\lambda\iv) M_\mu(a_\lambda)e^{-T(a_\lambda\iv)K_\mu(a_\lambda)}\Big)
\ee
under the conditions stated in Theorem \ref{thm2}(a). The convergence (\ref{E.conv})
is locally uniform in $\lambda$.  Hence $E(f,\lambda)$ is analytic in $\lambda$ and $E(f,0)=1$.
This implies that $E(f,\lambda)$  is a proper moment generating  function, and hence 
 $X_{f,n} - n f_0$ converges in distribution to some random variable $\mathcal{Z}$. That $\mathcal{Z}$
 has mean zero can be seen by differentiating (\ref{E.conv}) and putting $\lambda=0$
This concludes the first part  of Theorem \ref{thm2}. 

Notice that under the stronger conditions, the constant simplifies to
\be\label{Ela2}
   E(f, \lambda)=\det\Big( T(e^{-i\lambda(f-f_0)}) M_\mu(e^{i\lambda(f-f_0)})\Big)=
   \det\Big( T(e^{-i\lambda f}) M_\mu(e^{i\lambda f})\Big).
\ee
 What exactly $\mathcal{Z}$  is though is hard to understand from (\ref{Ela}) or (\ref{Ela2}). 
The following is the best we have; it completes the proof of Theorem \ref{thm2}.

\begin{proposition}\label{p6.5} 
Let $\beta>1$, $\sigma=1\vee 2/\beta$ and assume either $b\in F\ell^1_\sigma$ or $b\in F\ell^2_{\sigma+\varepsilon}$, $\eps>0$. Then there exists $\delta>0$ such that for $\lambda\in \C$ with $|\lambda|<\delta$ it holds that
\bq\label{Ea}
  \det ( T(e^{-\lambda b} ) M_{\mu}(e^{\lambda b}) )  &=&\exp\Big( \frac{\lambda^2}{2} \trace( H(b) H(\tb))+ \sum_{n=2}^\iy \frac{\lambda^n}{n !} \trace(B_n) \Big),
\eq
where the (trace class) operators $B_n$ are defined by the recursion
$$
B_{n+1}=M_{\mu}(b^{n+1})-\sum_{k=1}^n{ n \choose k} B_{n+1-k} M_{\mu}(b^k),\qquad n\ge 0.
$$
\end{proposition}

Ahead of  the proof, we write out the first couple $B_n$'s. 
With
$M_k=M_\mu(b^k)$ we obtain $B_1=M_1$,
\bq
B_2 &=& M_2-M_1^2,\label{B2} \nn \\
B_3 &=& M_3-2M_2M_1-M_1M_2+2 M_1^3,  \nn \\
B_4 &=& M_4-3M_3M_1-M_1M_3-2M_2^2+6M_2M_1^2+3M_1M_2M_1+3M_1^2M_2-6M_1^4. \nn
\label{B4}
\eq
When $\mu$ is the unit mass at 1, then $M_{\mu}(b)=T(b)$ and one has
that
$\log \det(T(e^{-\lambda b}) T(e^{\lambda b}))$ equals
$  \lambda^2 \,\trace( H(b)H(\tilde b))$ (according to the Szeg\"o-Widom limit theorem).
That is, we have the above expressions with $M_k$ replaced by
$T_k=T(b^k)$ while at the same time $\trace B_2=\trace(H(b)H(\tilde b))$ and $\trace B_m=0$ for all
$m\ge 3$.  (This means that the cumulants of $\mathcal{Z}$ of order three and higher are vanishing.)  Back in the general case, we can subtract from the $B_k$ given by the above formulas
the corresponding expressions for the special case $M_k=T_k$ and then  take traces.
Substituting $M_k=T_k+K_k$ with $K_k=K(b^k)$, yields
\bq
\trace(B_2) &=& \trace(H(b) H(\tilde{b})) - \trace(2 T_1K_1+K_1^2), \nn
\\
\trace(B_3) &=& -3\,\trace(K_2T_1+K_1T_2+K_2K_1) + 2\,\trace(3K^2_1T_1+3K_1T_1^2+K_1^3), \nn  \\
\trace(B_4) &=& -4\,\trace(T_3 K_1+K_3 T_1+K_3K_1)-2\,\trace(2T_2K_2+K_2^2)\nn\\
&&\mbox{}+12\,\trace(T_2T_1K_1+T_2K_1T_1+T_2K_1^2+K_2T_1^2+K_2T_1K_1+K_2K_1T_1+K_2K_1^2)\nn\\
&&\mbox{}-6\,\trace(4T_1^3K_1+4T_1^4K_1^2+2T_1K_1T_1K_1+4T_1K_1^3+K_1^4). \nn
\eq
All products under the traces are trace class operators and thus each of the above objects can be computed explicitly in terms of infinite sums.  Still, the expressions become increasingly 
intractable, and  we do not see how further simplifications are possible.

\begin{proof}
Set $a_\lambda=e^{\lambda b}$ and  split the determinant $E[a_\lambda]=\det T(a_\lambda\iv)M_\mu(a_\lambda)$ into two
parts $E[a_\lambda]=E_{1}(\lambda)E_{2}(\lambda)$ where
$$
E_1(\lambda)=\det T(\a\iv) e^{\lambda T(b)},
\quad E_2(\lambda)=\det e^{-\lambda T(b)}M_\mu(\a).
$$
First of all, both expressions are well defined because the expressions under the determinant are of the form identity plus trace class. Indeed, this has been shown  for $T(a_\lambda\iv)e^{\lambda T(b)}$ in \cite[Prop.~7.1]{Eh1}. Now observe that $M_\mu(a_\lambda)$ is a trace class perturbation of $T(a_\lambda)$.

{}It has been shown in \cite[Sec.~3]{Eh2} (see also the proof of Thm.~2.5 in \cite{BasorEhr})
that
$$
E_1(\lambda)=\exp\left(\frac{\lambda^2}{2} \trace (H(b)H(\tb))\right).
$$

It is straightforward to verify that $E_2(\lambda)$ depends analytically on $\lambda$ (see again \cite{Eh1,Eh2}).
Assume now that $|\lambda|$ is sufficiently small such that
$M_{\mu}(a_\lambda)$, being close to the identity operator, is invertible
and hence the determinants $E_2(\lambda)$ are nonzero. Notice that $E_2(0)=1$,
whence there is no problem of defining a logarithm in a small neighborhood of zero,
$$
f(\lambda) :=\log\det e^{-\lambda T(b)} M_{\mu}(\a).
$$

Recall that for invertible analytic operator-valued functions $F(\lambda)$ of the form identity plus trace 
class we have the well-known the formula 
$(\log\det F(\lambda))' = \trace F'(\lambda)F\iv(\lambda)$.
As a consequence, for invertible $A(\lambda)$ and $B(\lambda)$, whose product is identity plus
trace class, we have
\be\label{det_AB}
(\log\det A(\lambda) B(\lambda))' = \trace\Big(A\iv(\lambda) A'(\lambda) +B'(\lambda)B\iv(\lambda)\Big).
\ee
{}From this we obtain 
$$
f'(\lambda) = \trace\Big( M_{\mu}(\a)' M_{\mu}\iv(\a) -T(b)\Big).
$$
For small $|\lambda|$ introduce the well-defined analytic function $B(\lambda)$ defined by $B(0) =0$ and
$$
B'(\lambda)=M_\mu(a_\lambda)'M_\mu\iv(a_\lambda).
$$
Writing out this relation in terms of power series (with $B(\lambda)=\sum_{k=1}^\iy
\lambda^k B_k/k!$)
it follows that 
$$
\left( \sum_{k=0}^\iy \frac{\lambda^k B_{k+1}}{k!}\right)\left(\sum_{k=0}^\iy \frac{\lambda^k M_{\mu}(b^k)}{k!}\right)
=\sum_{n=0}^\iy \frac{\lambda^n M_{\mu}(b^{n+1})}{n!}.
$$
Inspection of the  $n$-th coefficient ($n\ge0$)  produces 
$$
M_{\mu}(b^{n+1})= B_{n+1}+\sum_{k=1}^n{ n \choose k} B_{n+1-k} M_{\mu}(b^k)
$$
which implies the recursion. Noting that $f(0)=0$, $B(0)=0$, and $f'(\lambda)=\trace(B'(\lambda)-T(b))$
yields
$$
E_2(\lambda)=\det e^{-\lambda M_{\mu}(b)}M_{\mu}(a_\lambda)=\exp(\trace(B(\lambda)-\lambda T(b))).
$$
Since we have  $B_1=M_{\mu}(b)$ from the recursion and $\trace K_\mu(b)=0$ ($\rho_{kk}=1$)  the proof is finished.
\end{proof}


\section{Limit theorems: the case $\boldsymbol{1/2<\beta\le1}$ (C2)}
\label{s7}

We will now prove the main results of the introduction related to the moment condition (C2).

Let us first prove Theorem \ref{thm3}(a).
Put $B=F\ell^2(\nu)$ with the conditions on $\nu$ stated there. It follows immediately that 
$B\subseteq F\ell^2_\sigma$ with $\sigma\ge1/\beta\ge 1$. 
Hence by Proposition \ref{p3.2} the Hilbert-Schmidt condition {\bf (HS)} holds.
Moreover, Proposition \ref{p4.1} implies that $B$ is a suitable Banach algebra.
Hence we can use Theorem \ref{mainthm}(b) and obtain (\ref{eq20}) with the constant $H[a]$.
We are left with determining the asymptotics of  the trace  of $P_n T(a\iv) K_\mu(a)P_n$, for which we can  use Theorem \ref{t5.4}(ii). Therein our Banach algebra is continuously embedded into the Banach space $F\ell^2_\sigma\cap F\ell^2(\nu)$ (with possibly different $\nu$).
With $b=a\iv$ the asymptotics equals $\Omega(a,a\iv)\cdot \iota_\mu(2 n)+C_\mu(a,a
\iv)+ o(1)$ with 
$$
\Omega[a]:=\Omega(a,a\iv)=-\frac{1}{4\pi}\int_0^{2\pi}a'(e^{it})(a\iv(e^{it}))'\, dt =\frac{1}{4\pi}\int_0^{2\pi} \left(\frac{a'(e^{it})}{a(e^{it})}\right)^2\, dt.
$$
This gives the correct constant in (\ref{Ga}). As for the constant $F[a]$ in (\ref{con.1})
we remark that 
\be
F[a]=e^{C_\mu(a,a\iv)}
\det\left( T(a\iv) M_\mu(a)e^{-T(a\iv)K_\mu(a)}\right)
\ee
where $C_\mu(a,a\iv)$ is given by (\ref{f.54}), but we make no attempt to simplify the expression.

Notice that both Theorem \ref{t5.4}(ii) and Proposition \ref{p4.1}(iv) require the rather complicated Banach algebra $B=F\ell^2(\nu)$.
This completes the proof of Theorem \ref{thm3}(b).

\medskip
Let us now turn to the proof of Theorem \ref{thm1}. We assume that $B=F\ell^2_{\sigma}$ with $\sigma=1/\beta\vee 3/(2\gamma)$. 
There is no change in the applicability of Theorem \ref{mainthm}(b), however, the function to which we apply it is appropriately re-scaled.
In particular, it depends on $n$, and therefore the statements about uniform convergence are needed.

Let us first point out that the mean of $X_{f,n}$ is precisely $n f_0$ and the variance is asymptotically $\iota_\mu(2n)$ times a scaled $F \ell_1^2$-norm of $f$. (This will actually follow from Theorem \ref{thm1}, but can also be shown by a direct computation resembling the one in Section \ref{s5}.) This motivates to replace $X_{f,n}$ with $f\in B$ by
\be\label{f.75}
X_{f,n}^{\rm scal}:= \frac{X_{f,n}-n f_0}{\sqrt{\iota_\mu(2n)}} 
=X_{g_n,n}, \qquad  g_{n}(e^{ix}):=\frac{f(e^{ix})- f_0}{\sqrt{\iota_\mu(2n)}}.
\ee
Then using (\ref{detform})
$$
\E_{\m,n}[e^{i\lambda X_{f,n}^{\rm scal}}]=\det M_{\mu,n}(a_{\lambda,n})
$$
with $ a_{\lambda,n}=e^{i\lambda g_{n}}$.  Because $\iota_\mu(2n)\to\iy$, the elements $g_n$ ($n\in\N$)  lie in  
a compact subset of $B$, and so 
 $a_{\lambda,n}$ lie in a compact subset of $\Phi(B)$ (see also Proposition \ref{p4.2}(iii)).

By Theorem \ref{mainthm}(b)
$$
\lim_{n\to\iy}
\frac{\det M_{\mu,n}(a_{\lambda,n})}{G[a_{\lambda,n}]^n\cdot
\exp(\trace P_n T(a_{\lambda,n}\iv) K_\mu(a_{\lambda,n})P_n)}=\lim_{n\to\iy} H[a_{\lambda,n}],
$$
due to uniform convergence on compact subsets. The regularized determinant $H[a_{\lambda,n}]$ converges to 
$H[1]=1$ since $T(a_{\lambda,n}\iv)K_\mu(a_{\lambda,n})\to T(1)K_\mu(1)=0$ in the Hilbert-Schmidt norm.
Here we have to use Proposition \ref{p3.3} and the estimate implied by {\bf (HS)}.

Again by Proposition \ref{p4.2}(iii),  $G[a_{\lambda,n}]=1$.  To evaluate  the trace we will used Theorem \ref{t5.4}(i).
Define
$$
h=i\lambda(f-f_0) \quad\mbox{ and }\quad
s_n=\sqrt{\iota_\mu(2n)}
$$
and introduce the functions $p_n,q_n\in B$ via series expansion
\bq
a_{\lambda,n} = e^{h/s_n}=1+h/s_n+p_n/s_n^2, \  \ 
a_{\lambda,n}\iv = e^{-h/s_n}=1-h/s_n+q_n/s_n^2.\nn
\eq
Notice immediately that $p_n\to h^2/2$ and $q_n\to h^2/2$ in the norm of $B$.
Denoting $t_n(b,a)=\trace (P_nT(b)K_\mu(a)P_n)$ we have that
$$
t_n(a_{\lambda,n}\iv,a_{\lambda,n})=
-\frac{t_n(h,h)}{s_n^2}+\frac{-t_n(h,p_n)+t_n(q_n,h)+s_n\iv t_n(p_n,q_n)}{s_n^3}
$$
because in general $t_n(b,1)=t_n(1,a)=0$. 
Theorem \ref{t5.4}(i) says that for $a,b\in B$ we have $t_n(b,a)=\Omega(a,b) s_n^2+o(s_n^2)$
and that the convergence is uniform on compact sets.
Hence, applying this to all of the above expressions involving $t_n$ and using that 
 $p_n$ and $q_n$ are from compact subsets of $B$, it follows that 
$$
 \lim_{n \rightarrow \infty} 
 t_n(a_{\lambda,n}\iv,a_{\lambda,n}) 
 =-\Omega(h,h)
=-\frac{\lambda^2}{2}\sum_{k=-\iy}^\iy k^2 f_k f_{-k}=-
\frac{\lambda^2}{4\pi}\int_0^{2\pi} (f'(e^{ix}))^2\, dx.
$$
This implies 
\be\label{f.61}
\lim_{n\to\iy} \E_{\m,n}[e^{i\lambda X_f^{\rm scal}}]
=\exp\left(-\frac{\lambda^2}{2}\sum_{k\in\Z} k^2 f_k f_{-k}\right)=\exp\left(-\frac{\lambda^2}{4\pi}
\int_{0}^{2\pi} (f'(e^{ix}))^2\, dx\right)
\ee
completing the proof of Theorem \ref{thm1}.
Moreover, it is easy to see that the convergence (\ref{f.61}) is uniform on bounded sets of $\lambda$ and compact sets of $f\in B$.


\section*{Appendix: On the Toeplitz $\circ$ Hankel formula}

We wish to compute the integral
$$
   \mathcal{I}_{\m,n}(\varphi) =   \frac{1}{Z_{\m,n}} \int_{\C^n}  \prod_{k=1}^n \varphi( \arg{z_k})  \prod_{k < \ell}  |z_k - z_{\ell} |^2    \prod_{k=1}^n  d{\m}(z_k),
$$
where  $d\m$ is radial ($d\m(z) = d\theta  d \mu(r)$, $z=r e^{i\theta}$)  and $Z_{\m,n}$ is  chosen so that $\mathcal{I}_{\m,n}(1) = 1$.

To begin, write
$$
 \prod |z_k - z_{\ell} |^2  = \det \Bigl[  [ z_k^{\ell-1} ]  \cdot [ \bar{z}_k^{\ell-1} ]^T  \Bigr]
$$
where  $[ z_k^{\ell-1} ] $ denotes the $n \times n$ matrix with $z_k^{\ell -1}$ in row $k$ and column $\ell$. That is to say,
$$
   \prod |z_k - z_\ell |^2  = \det \left[ \begin{array}{cccc} n  & \sum_{k=1}^n \bar{z_k}  & \sum_{k=1}^n \bar{z}_k^2  & \dots \\
                                                       \sum_{k=1}^n z_k & \sum_{k=1}^n z_k  \bar{z}_k & \sum_{k=1}^n z_k \bar{z}_k^{2} & \dots \\
                                                       \vdots & \vdots & \vdots & \ddots \end{array} \right].
$$
Now  expand the first column on the right hand side via the linearity of the determinant, writing it as sum of $n$ determinants with first
column $[1, z_k, z_k^2, \dots, z_k^{n-1} ]$.  By the product structure of $\prod \varphi(\arg z_k)  d \m( z_k) $ each of the resulting $n$ integrals 
are the same. Thus, we can replace the $\prod |z_k - z_\ell  |^2$ in the measure with
$$
  \det \left[ \begin{array}{cccc} 1  & \sum_{k=1}^n \bar{z_k}  & \sum_{k=1}^n \bar{z}_k^2  & \dots \\
                                                         z_1 & \sum_{k=1}^n z_k \bar{z}_k & \sum_{k=1}^n z_k \bar{z}_k^{2} & \dots \\
                                                       \vdots & \vdots & \vdots & \ddots \end{array} \right]                                                      
    =    \det \left[ \begin{array}{cccc} 1  & \sum_{k=2}^n \bar{z_k}  & \sum_{k=2}^n \bar{z}_k^2  & \dots \\
                                                         z_1 & \sum_{k=2}^n z_k \bar{z}_k & \sum_{k=2}^n z_k \bar{z}_k^{2} & \dots \\
                                                       \vdots & \vdots & \vdots & \ddots \end{array} \right],                                                
$$
at the cost of introducing a constant factor which may be absorbed into $Z_{\m,n}$.
This procedure may be repeated, and after the $n$-th iteration we conclude that
\begin{eqnarray*}
 {\mathcal{I}}_{\m,n}(\varphi) &  =  &  \frac{1}{Z_{\m,n}} \int_{\C^n}
       \prod_{k=1}^n \varphi(\arg z_k)   \det \Bigl[  z_k^{\ell-1} \bar{z}_k^{k-1}  \Bigr]_{1 \le k, \ell \le  n}   \prod_{k=1}^n d \m(z_k)  \\
                                          & = &    \frac{1}{\tilde{Z}_{\m,n}} \det \Bigl[ \frac{1}{2\pi} \int_{\C} \varphi(\arg z )  z^{\ell} \bar{z}^{k}  d\m(z)  \Bigr]_{0 \le k,\ell \le n-1},              
\end{eqnarray*}
after using the linearity of the determinant once more.  And, as
$$
\frac{1}{2\pi}  \int_{\C} \varphi(\arg z )  z^{\ell} \bar{z}^{k}  d\m(z)   = \varphi_{k-\ell}  \int_0^{\infty} r^{k+ \ell }  d \mu(r)=\varphi_{k-\ell} m_{k+\ell} ,
$$
setting $\varphi \equiv 1$ we find that $\tilde{Z}_{\m,n} =\prod_{k=0}^{n-1} m_{2k}$, and so formula (\ref{detform}).

\bigskip

\noindent{\bf{Acknowledgments}}  The work of the first named author was supported in part by NSF grant DMS-0901434;
that of the second  by NSF grant DMS-0645756.

\end{document}